\documentclass{amsart}

\usepackage[all]{xy}

\theoremstyle{plain}
\newtheorem{thm}{Theorem}[section]
\newtheorem{prop}[thm]{Proposition}
\newtheorem{lemma}[thm]{Lemma}

\newtheorem{cor}[thm]{Corollary}

\theoremstyle{definition}
\newtheorem{dfn}[thm]{Definition}

\newtheorem{hypo}[thm]{Hypothesis}

\theoremstyle{remark}
\newtheorem{rem}[thm]{Remark}

\newcommand{\HH}{\mathrm{H}}

\newcommand{\D}{\mathrm{D}}
\newcommand{\SD}{\mathrm{SD}}
\newcommand{\Q}{\mathrm{Q}}

 \newcommand{\unter}[2]{\genfrac{}{}{0pt}{}{#1}{#2}}

\begin{document}

\title[Brauer's work and universal deformation rings]{Brauer's generalized decomposition numbers and universal deformation rings}

\author{Frauke M. Bleher}
\thanks{The author was supported in part by  
NSA Grant H98230-11-1-0131.}

\begin{abstract}
The versal deformation ring $R(G,V)$ of a mod $p$ representation $V$ of a profinite group
$G$ encodes all isomorphism classes of lifts of $V$ to representations of $G$ over complete local 
commutative Noetherian
rings.  We introduce a new technique for determining $R(G,V)$ when $G$ is finite which 
involves Brauer's generalized decomposition numbers.
\end{abstract}

\maketitle


\section{Introduction}
\label{s:intro}

Let $k$ be an algebraically closed field of positive characteristic $p$, let $W=W(k)$ be the
ring of infinite Witt vectors over $k$, and let $G$ be a finite group. 
An important question in the representation theory of $G$ is whether a finitely generated 
$kG$-module $V$ can be lifted to $W$.  For example, Green's lifting theorem 
shows that this is possible if there are no non-trivial $2$-extensions of $V$ by itself. 
A natural generalization of this question is to consider the functor  which sends 
each complete local commutative Noetherian ring $R$ with residue field $k$ to the set of
isomorphism classes of lifts of $V$ over $R$.  If this functor is represented by
a ring $R(G,V)$, we say that $R(G,V)$ is the universal deformation ring of $V$.
More generally, one can always associate to $V$ a versal deformation ring
$R(G,V)$, whose precise definition is recalled in Section \ref{s:prelim}.  
It was shown in \cite[Prop. 2.1]{bc} that if the stable endomorphism ring 
$\underline{\mathrm{End}}_{kG}(V)$ is isomorphic to $k$, then
the versal deformation ring $R(G,V)$ is always universal.

Apart from the fact that universal deformation rings of representations of finite groups
give more insight into the representation theory of these finite groups, there is another
important motivation for studying these universal deformation rings. 
Namely,
universal deformation rings for finite groups 
provide a good test case for various conjectures concerning the ring theoretic properties
of universal deformation rings for profinite Galois groups. 
For example, Flach asked whether there could be universal deformation rings which 
are not complete intersections (see \cite{flach}). 
In \cite{bc4.9,bc5} it was shown that the universal deformation ring of the non-trivial
irreducible mod $2$ representation of the symmetric group $S_4$ is not a complete intersection;
in fact it is not even Cohen-Macaulay.
This led to infinitely many examples of real quadratic fields $L$ such that the universal deformation 
ring of the inflation of this representation to the  Galois group over $L$ of the 
maximal totally unramified extension of $L$ is not a complete intersection.
In \cite{bcs}, examples of finite groups $G$ and mod $p$ representations of $G$
for every odd prime number $p$ were constructed such that the universal deformation rings of 
these representations are not complete intersections.
The main advantage of computing universal deformation rings for representations of finite groups
is that one can use deep results from modular representation theory due to Brauer, 
Erdmann \cite{erd}, Linckelmann \cite{linckel,linckel1}, Carlson-Th\'{e}venaz \cite{carl2,carl1.5}, 
and others. 

In this paper, we propose a new method of determining universal deformation rings, using Brauer's
generalized decomposition numbers. Brauer generalized the usual decomposition numbers 
in \cite{brauerordinaryandmodular}
to be able to express the values of ordinary irreducible characters not only on $p$-regular elements
but also on $p$-singular elements
in terms of Brauer characters. We will show how to use these generalized decomposition numbers to
determine the universal deformation rings $R(G,V)$ for certain mod $p$ representations $V$ of finite groups 
$G$ whose stable endomorphisms are all given by scalars. 
The $V$ we consider are those for which Brauer's generalized decomposition numbers 
carry the most information;  these $V$ are called maximally ordinary below. 
For maximally ordinary $V$, the generalized decomposition numbers 
enable us to find a family of Galois orbits of ordinary irreducible characters of $G$
which can be used to construct lifts of $V$
to large local rings in characteristic 0. This provides a powerful tool for computing universal
deformation rings. 
The need for the use of Brauer's generalized decomposition numbers is related to how much
fusion of conjugacy classes of $p$-power order elements occurs in $G$.

Suppose $V$ is a $kG$-module whose stable endomorphism ring is isomorphic to $k$. Then
$V$ has a unique non-projective indecomposable summand (up to isomorphism). Since 
the universal deformation ring $R(G,V)$ only depends on this indecomposable direct summand
by \cite[Cor. 2.7]{bc},
we may assume that $V$ is indecomposable. Hence there exists a unique $p$-block $B$ to which
$V$ belongs. The case when $B$ has finite representation type has been 
fully studied in \cite{bc}. 
Therefore, we concentrate in this paper on the case when $B$ has infinite tame 
representation type. Because the study of local blocks requires different arguments, as was demonstrated 
for example in \cite{diloc}, we further assume that $B$ is not local. 
Let $D$ be a defect group of $B$.
We say $V$ is 
maximally ordinary
if the Brauer character of $V$ is the restriction of an ordinary irreducible character $\chi$ 
such that for every $\sigma\in D$ of maximal $p$-power order,
Brauer's generalized decomposition numbers corresponding to $\sigma$ and $\chi$
do not all lie in $\{0,\pm 1\}$ (see Definition \ref{def:maximallyordinary}).
This condition ensures that Brauer's generalized decomposition numbers carry enough 
information for our method to be applied to $V$.

The following theorem summarizes our main results;  more precise statements can be found in 
Sect. \ref{s:udr}, and in particular in Corollary \ref{cor:biglist} and Theorem \ref{thm:main}.

\begin{thm}
\label{thm:supermainNEW}
Suppose $G$ is a finite group, $B$ is a non-local block of $kG$ of infinite tame representation type,
and $D$ is a defect group of $B$ of order $p^n$. There exists an indecomposable $kG$-module 
$V$ belonging to $B$ whose stable endomorphism ring is isomorphic to $k$ and which
is maximally ordinary
if and only if $n\ge 4$. 
Moreover, the isomorphism class of every such $V$ can be described explicitly.
Suppose $V$ is such a maximally ordinary module. 
There exists a monic polynomial $q_n(t)\in W[t]$ of degree $p^{n-2}-1$
which depends on $D$ but not on $V$ and which can be given explicitly
such that either
\begin{enumerate}
\item[(i)] $R(G,V)/pR(G,V)\cong k[[t]]/(t^{p^{n-2}-1})$, in which case $R(G,V)\cong W[[t]]/(q_n(t))$, or
\item[(ii)] $R(G,V)/pR(G,V)\cong k[[t]]/(t^{p^{n-2}})$, in which case  $R(G,V)\cong
W[[t]]/(t\,q_n(t),p\,q_n(t))$.
\end{enumerate}
In all cases, the ring $R(G,V)$ is isomorphic to a subquotient ring of $WD$, and it is
a complete intersection if and only if we are in case $(i)$.
\end{thm}
In particular, Theorem \ref{thm:supermainNEW} gives a positive answer to \cite[Question 1.1]{bc} 
for all $B$ and $V$ considered in the theorem.
A precise description of the maximally ordinary modules $V$ belonging to $B$ is given in 
Corollary \ref{cor:biglist}.
A formula for the polynomials $q_n(t)$ can be found
in Definition \ref{def:seemtoneed} and Remark \ref{rem:ohyeah}.

The use of Brauer's generalized decomposition numbers provides a correction to some arguments
in \cite[Sect. 3.4]{3sim}, \cite[Sect. 5]{3quat}, and \cite[Sect. 3.2]{2sim}. Namely, in these papers a 
formula for the values of the ordinary irreducible characters on elements
in $D$ of maximal $p$-power order was assumed, which is true for principal infinite tame blocks but
cannot be verified for arbitrary infinite tame blocks because there may be more fusion
of $D$-conjugacy classes in $G$ when the blocks are not principal 
(see Remark \ref{rem:principal}).

The paper is organized as follows. In Section \ref{s:prelim}, we recall the definitions of deformations and 
deformation rings.
In Section \ref{s:braueristheman}, we give a brief introduction to Brauer's generalized decomposition 
numbers, as introduced in \cite{brauerordinaryandmodular,brauerdarst2}.
Using \cite{erd}, we descibe in Section \ref{s:tame} the quivers and relations of the basic algebras
of all non-local blocks $B$ of infinite tame representation type,
and provide their decomposition matrices. 
In Section \ref{s:ordinarytame}, we describe results of 
\cite{brauer2,olsson}
about the ordinary irreducible characters of $G$ belonging to $B$.
In Section \ref{s:udr}, we prove Theorem \ref{thm:supermainNEW}.


\section{Preliminaries}
\label{s:prelim}

In this section, we give a brief introduction to versal and universal deformation rings and deformations. 
For more background material, we refer the reader to \cite{maz1} and \cite{lendesmit}.

Let $k$ be an algebraically closed field of characteristic $p>0$, and let $W$ be the ring of infinite 
Witt vectors over $k$.
Let $\hat{\mathcal{C}}$ be the category of all complete local commutative Noetherian 
rings with residue field $k$. Note that all rings in $\hat{\mathcal{C}}$ have a natural $W$-algebra structure.
The morphisms in $\hat{\mathcal{C}}$ are continuous $W$-algebra 
homomorphisms which induce the identity map on $k$.

Suppose $G$ is a finite group and $V$ is a finitely generated $kG$-module. 
A \emph{lift} of $V$ over an object $R$ in $\hat{\mathcal{C}}$ is a pair $(M,\phi)$ where $M$ is a finitely 
generated $RG$-module which is free over $R$, and $\phi:k\otimes_R M\to V$ is an isomorphism of 
$kG$-modules. Two lifts $(M,\phi)$ and $(M',\phi')$ of $V$ over $R$ are isomorphic if there is an 
isomorphism $f:M\to M'$ with $\phi=\phi'\circ (k\otimes f)$. The isomorphism class $[M,\phi]$ of a lift 
$(M,\phi)$ of $V$ over $R$ is called a \emph{deformation} of $V$ over $R$, and the set of all such deformations 
is denoted by $\mathrm{Def}_G(V,R)$. The deformation functor
$$\hat{F}_V:\hat{\mathcal{C}} \to \mathrm{Sets}$$ 
is a covariant functor which
sends an object $R$ in $\hat{\mathcal{C}}$ to $\mathrm{Def}_G(V,R)$ and a morphism 
$\alpha:R\to R'$ in $\hat{\mathcal{C}}$ to the map $\mathrm{Def}_G(V,R) \to
\mathrm{Def}_G(V,R')$ defined by $[M,\phi]\mapsto [R'\otimes_{R,\alpha} M,\phi_\alpha]$, where  
$\phi_\alpha=\phi$ after identifying $k\otimes_{R'}(R'\otimes_{R,\alpha} M)$ with $k\otimes_R M$.

Suppose there exists an object $R(G,V)$ in $\hat{\mathcal{C}}$ and a deformation 
$[U(G,V),\phi_U]$ of $V$ over $R(G,V)$ with the following property:
For each $R$ in $\hat{\mathcal{C}}$ and for each lift $(M,\phi)$ of $V$ over $R$ there exists 
a morphism $\alpha:R(G,V)\to R$ in $\hat{\mathcal{C}}$ such that $\hat{F}_V(\alpha)([U(G,V),\phi_U])=
[M,\phi]$, and moreover $\alpha$ is unique if $R$ is the ring of dual numbers
$k[\epsilon]/(\epsilon^2)$. Then $R(G,V)$ is called the \emph{versal deformation ring} of $V$ and 
$[U(G,V),\phi_U]$ is called the \emph{versal deformation} of $V$. If the morphism $\alpha$ is
unique for all $R$ and all lifts $(M,\phi)$ of $V$ over $R$, 
then $R(G,V)$ is called the \emph{universal deformation ring} of $V$ and $[U(G,V),\phi_U]$ is 
called the \emph{universal deformation} of $V$. In other words, $R(G,V)$ is universal if and only if
$R(G,V)$ represents the functor $\hat{F}_V$ in the sense that $\hat{F}_V$ is naturally isomorphic to 
the Hom functor
$\mathrm{Hom}_{\hat{\mathcal{C}}}(R(G,V),-)$.

Note that the above definition of deformations can be weakened as follows.
Given a lift $(M,\phi)$ of $V$ over a ring $R$ in 
$\hat{\mathcal{C}}$, define the corresponding weak deformation to be
the isomorphism class of $M$ as an $RG$-module, without taking into account the specific 
isomorphism $\phi:k\otimes_RM\to V$. 
In general, a weak deformation of $V$ over $R$ identifies more lifts than a deformation of $V$ 
over $R$
that respects the isomorphism $\phi$ of a representative $(M,\phi)$.
However, if the stable
endomorphism ring $\underline{\mathrm{End}}_{kG}(V)$ is isomorphic to $k$, these two 
definitions
of deformations coincide (see  \cite[Remark 2.1]{3quat}).

By \cite{maz1}, every finitely generated $kG$-module $V$ has a versal deformation ring.
Since $G$ is a finite group, we have the following sufficient criterion for the
universality of $R(G,V)$:

\begin{prop}
\label{prop:stablendudr}
{\rm (\cite[Prop. 2.1]{bc})}
Suppose $V$ is a finitely generated $kG$-module whose stable endomorphism ring 
$\underline{\mathrm{End}}_{kG}(V)$ is isomorphic to $k$.  Then $V$ has  a universal 
deformation ring $R(G,V)$.
\end{prop}


\section{Brauer's generalized decomposition numbers}
\label{s:braueristheman}

In this section, we give a brief introduction to Brauer's generalized decomposition numbers,
emphasizing the results needed in this paper. Throughout this section,
let $p$ be a fixed prime number and let $G$ be a finite group such that
$p$ divides $\#G$. Let $P$ be a fixed Sylow $p$-subgroup of $G$.

Brauer introduced generalized decomposition numbers in \cite{brauerordinaryandmodular}
to be able to express the values of the ordinary irreducible characters of  $G$ on all 
conjugacy classes by means of the irreducible $p$-modular characters of certain
subgroups of $G$. More precisely, let $g\in G$ and write $g$ (uniquely) as
$g=uv$ where $u$ is a $p$-element of order $p^\alpha$ and $v$ is a
$p$-regular element in the centralizer $C_G(u)$ (where we allow $u$ or $v$ to be  identity elements). 
Let $\mathrm{IBr}(C_G(u))$ denote the set of distinct irreducible $p$-modular characters of $C_G(u)$.
By \cite[Sect.  1]{brauerordinaryandmodular}, if $\chi$ is an ordinary irreducible 
character of $G$ then we have a formula
\begin{equation}
\label{eq:brauer0}
\chi(uv) = \sum_{\varphi\in\mathrm{IBr}(C_G(u))} d_{\chi,\varphi}^{\,(u)}\;\varphi(v)
\end{equation}
where 
the numbers $d_{\chi,\varphi}^{\,(u)}$, $\varphi\in\mathrm{IBr}(C_G(u))$,
are algebraic integers in the field of $p^\alpha$-th
roots of unity which do not depend on $v$. These are called the \emph{generalized decomposition 
numbers} of $G$ corresponding to $u$ and $\chi$.

If $u=1_G$, then the generalized decomposition numbers corresponding to $u$ 
are the usual decomposition numbers, which are non-negative integers, and
(\ref{eq:brauer0}) is the usual formula for the restriction to 
$p$-regular elements of an ordinary irreducible character $\chi$ in terms of an integral combination of 
the irreducible $p$-modular characters of $G=C_G(1_G)$.

We can use these generalized decomposition numbers to express the ordinary character
table of $G$ as the product of two square matrices as follows.
Recall that we fixed a Sylow $p$-subgroup $P$ of $G$. 
Let $u_0,u_1\ldots,u_h$ be a complete system of
representatives of $G$-conjugacy classes of $p$-elements in $G$ with $u_0=1_G$
and $u_i\in P$ for all $1\le i\le h$.
For each $0\le i\le h$, let $v_{i,1},\ldots,v_{i,\ell_i}$ be a complete system of representatives
of $C_G(u_i)$-conjugacy classes of $p$-regular elements in $C_G(u_i)$ with $v_{i,1}=1_G$.
Then $\{u_iv_{i,j}\;|\; 0\le i\le h,1\le j\le \ell_i\}$ is a complete set of representatives
of the conjugacy classes of $G$. Moreover, for each $0\le i\le h$, there are precisely
$\ell_i$ irreducible $p$-modular characters of $C_G(u_i)$, which we denote by
$\varphi_{i,1},\ldots,\varphi_{i,\ell_i}$. Let $\{\chi_1,\ldots,\chi_c\}$ be a complete
set of representatives of the ordinary irreducible characters of $G$. Then Brauer's above formula
(\ref{eq:brauer0}) can be written as
\begin{equation}
\label{eq:brauer1}
\chi_s(u_iv_{i,j})= \sum_{t=1}^{\ell_i} d_{s,t}^{\,i} \;\varphi_{i,t}(v_{i,j})
\end{equation}
for all $1\le s\le c$, $0\le i\le h$, $1\le j\le \ell_i$, where
$d_{s,t}^{\,i} = d_{\chi_s,\,\varphi_{i,t}}^{\,(u_i)}$ for all $1\le t\le \ell_i$. 
Write the conjugacy class representatives in the order
$$u_0v_{0,1},\ldots,u_0v_{0,\ell_0},u_1v_{1,1},\ldots,u_1v_{1,\ell_1},\ldots,
u_hv_{h,1},\ldots,u_hv_{h,\ell_h}.$$
Using (\ref{eq:brauer1}),
we obtain that the ordinary character table $\mathcal{X}$ can be written as a product
$\mathcal{X} = \Delta\cdot \Phi$, where $\Delta$ contains the generalized decomposition 
numbers and $\Phi$ is a block diagonal matrix
\begin{equation}
\label{eq:Phi}
\Phi=\left(\begin{array}{ccc}
\Phi_0&&0\\
&\ddots&\\
0&&\Phi_h\end{array}\right)
\end{equation}
with
$$\Phi_i=\left(\begin{array}{ccc}
\varphi_{i,1}(v_{i,1})&\cdots&\varphi_{i,1}(v_{i,\ell_i})\\
\vdots&&\vdots\\
\varphi_{i,\ell_i}(v_{i,1})&\cdots&\varphi_{i,\ell_i}(v_{i,\ell_i})
\end{array}\right)$$
for all $0\le i\le h$. By \cite[Sect. 7, p. 45]{brauerdarst2}, the square of the determinant
of $\Delta$ is $\pm p^a$ for some $a\in\mathbb{Z}^+$, and the square  of the
determinant of $\Phi$ is an integer which is relatively prime to $p$.

Fix now a $p$-modular block $B$ of $G$ and suppose that there are
$k(B)$ ordinary irreducible characters belonging to $B$.
Reorder $\chi_1,\ldots,\chi_c$ such that $\chi_1,\ldots,\chi_{k(B)}$
belong to $B$. For each $0\le i\le h$, reorder $\varphi_{i,1}\ldots,
\varphi_{i,\ell_i}$ such that 
the first $m_i$  characters, $\varphi_{i,1},\ldots,\varphi_{i,m_i}$, are precisely the
irreducible $p$-modular characters belonging to blocks of 
$C_G(u_i)$ whose Brauer correspondents in $G$ are equal to $B$.
(Note that for each block $b$ of $C_G(u_i)$, its Brauer correspondent $b^G$ in $G$ is
well-defined since the centralizer in $G$ of a defect group of $b$ is contained
in $C_G(u_i)$, see \cite[Sect. 14]{alp}.) It follows from \cite[Sect. 6]{brauerdarst2}
that for $1\le s\le k(B)$, (\ref{eq:brauer1}) can be rewritten as
\begin{equation}
\label{eq:brauer2}
\chi_s(u_iv_{i,j}) = \sum_{t=1}^{m_i} d_{s,t}^{\,i} \;\varphi_{i,t}(v_{i,j})
\end{equation}
for all $0\le i\le h$, $1\le j\le \ell_i$. 
In particular, we obtain for all $0\le i\le h$ that
\begin{equation}
\label{eq:brauer3}
\mathcal{X}_{B,i}=\left(\begin{array}{ccc}
\chi_1(u_iv_{i,1})&\cdots&\chi_1(u_iv_{i,\ell_i})\\
\vdots&&\vdots\\
\chi_{k(B)}(u_iv_{i,1})&\cdots&\chi_{k(B)}(u_iv_{i,\ell_i})
\end{array}\right)=
\Delta_{B,i}\cdot \Phi_i
\end{equation}
where
$$\Delta_{B,i}=\left(\begin{array}{lllccc}
d_{1,1}^{\,i}&\cdots&d_{1,m_i}^{\,i}&0&\cdots&0\\
\quad\vdots&&\quad\vdots&\vdots&&\vdots\\
d_{k(B),1}^{\, i}&\cdots&d_{k(B),m_i}^{\, i}&0&\cdots&0
\end{array}\right)$$
and we assume, as above, that $\varphi_{i,1}\ldots,\varphi_{i,m_i}$ are precisely the
irreducible $p$-modular characters belonging to blocks of 
$C_G(u_i)$ whose Brauer correspondents in $G$ are equal to $B$.

Let now $k$ be an algebraically closed field of characteristic $p$ and view
$B$ as a block of $kG$. Let $W=W(k)$ be the ring of infinite Witt vectors over $k$ and
let $F$ be the fraction field of $W$. In particular, $W$ contains all roots of unity of order not 
divisible by $p$. Let $\overline{F}$ be a fixed algebraic closure of $F$,
and let $\xi$ be a root of unity in $\overline{F}$ whose order is a power of $p$
such that $F(\xi)$ is a splitting field for $G$.
Then $W[\xi]$ is the ring of integers of $F(\xi)$, and we can view the ordinary character 
table $\mathcal{X}$ of $G$ and the matrix $\Delta$ of generalized decomposition numbers
as taking values in $W[\xi]$, and the matrix $\Phi$ in (\ref{eq:Phi}) as taking values in $W$.
Since the square of the determinant of $\Phi$ is an integer that is relatively prime to $p$, it
follows that the determinant of $\Phi$, and hence the determinant of $\Phi_i$ for
all $0\le i\le h$, is a unit in $W$. Hence we can solve (\ref{eq:brauer3}) for
$\Delta_{B,i}$ to obtain
\begin{equation}
\label{eq:brauer4}
\Delta_{B,i} = \mathcal{X}_{B,i}\cdot \Psi_i
\end{equation}
for all $0\le i\le h$, where $\Psi_i=\Phi_i^{-1}$ is a matrix with values in $W$.
In particular, (\ref{eq:brauer4}) shows that if we reduce our discussion to a single block $B$ then we can
replace $\xi$ by a $p$-power order root of unity $\zeta$ in $\overline{F}$ such that all ordinary irreducible 
characters of $G$ belonging to $B$ are realizable over $F(\zeta)$, i.e. they correspond to
absolutely irreducible $F(\zeta)G$-modules.

\begin{rem}
\label{rem:fusion}
Equations (\ref{eq:brauer0}) and (\ref{eq:brauer2}) can be rewritten to reflect the influence of
fusion of $P$-conjugacy classes in $G$ (see  \cite[Sect. 1]{brblocks}). 
As before, let $u$ be a $p$-element of $G$ and let 
$v$ be a $p$-regular element in $C_G(u)$. Assuming the notation of the previous paragraph,
let $\chi$ be an irreducible $F(\xi)$-character which belongs to a block $B$ of $kG$.
Recall that a subsection $(y,b_y)$ for $B$ is a pair consisting of a $p$-element $y$ of $G$
and a block $b_y$ of $C_G(y)$ with $b_y^G=B$. We obtain
\begin{equation}
\label{eq:fusion!}
\chi(uv) = \sum_{(y,b_y)} \;\sum_{\varphi\in\mathrm{IBr}(b_y)}d_{\chi,\varphi}^{\,(y)} \;\varphi(z_yvz_y^{-1})
\end{equation}
where $(y,b_y)$ ranges over a system of representatives for the conjugacy classes of  subsections
for $B$ such that $y$ is conjugate to $u$ in $G$, say $u=z_y^{-1}yz_y$. For each
$(y,b_y)$, $\varphi$ ranges over the irreducible $p$-modular characters
associated with $b_y$.
\end{rem}

\begin{dfn}
\label{def:maximallyordinary}
Let $k$ be an algebraically closed field of characteristic $p$,  let $W$ be the ring of infinite Witt vectors 
over $k$, and let $\overline{F}$ be a fixed algebraic closure of the fraction field $F$ of $W$.
Suppose $B$ is a block of $kG$ and that $D$ is a defect group of $B$.
Let $\zeta\in \overline{F}$ be a root of unity of $p$-power order such that all ordinary irreducible 
characters of $G$ belonging to $B$ are realizable over $F(\zeta)$.
In particular, the ordinary (resp. $p$-modular) characters belonging to $B$ can be viewed
as taking values in $W[\zeta]$ (resp. $W$).
Suppose $V$ is an indecomposable $kG$-module whose stable endomorphism ring is isomorphic to $k$
and which belongs to $B$. We say $V$ is \emph{maximally ordinary}
if the $p$-modular character of $V$ is the restriction to the $p$-regular elements of an ordinary
irreducible character $\chi$ such that 
for every $\sigma\in D$ of maximal $p$-power order there exists an 
irreducible $p$-modular character $\varphi$ of $C_G(\sigma)$
such that 
$d_{\chi,\varphi}^{\,(\sigma)} \not\in\{0,\pm 1\}$.
\end{dfn}


\section{Tame blocks}
\label{s:tame}

For the remainder of this paper, we make the following assumptions:

\begin{hypo}
\label{hyp:alltheway}
Let $k$ be an algebraically closed field of positive characteristic $p$, let 
$W=W(k)$ be the ring of infinite Witt vectors over $k$, let $F$ be the fraction field of 
$F$, and let $\overline{F}$ be a fixed algebraic closure of $F$. 
Suppose $G$ is a finite group, 
$B$ is a non-local block of $kG$ of infinite tame representation type, and $D$ is a defect group 
of $B$ of order $p^n$. 
Let $\zeta$ be a primitive $p^{n-1}$-th root of unity in $\overline{F}$. 
\end{hypo}

It follows from \cite{bd,br,hi} that $p=2$, $n\ge 2$, and $D$ is dihedral, semidihedral or generalized 
quaternion. 
In particular, $n\ge 2$ if $D$ is dihedral, $n\ge 3$ if $D$ is generalized quaternion, 
and $n\ge 4$ if $D$ is semidihedral.
Moreover, Brauer  and Olsson proved that 
all ordinary irreducible characters of $G$ belonging to $B$ take values in $F(\zeta)$
(see \cite[Prop. (7D)]{brIV}, \cite[Prop. (5A)]{brauer2} and \cite[Prop. 4.1]{olsson}).
In fact, all ordinary irreducible characters of $G$ belonging 
to $B$ can be realized by simple $F(\zeta)G$-modules (see Section \ref{s:ordinarytame}).

Assume Hypothesis \ref{hyp:alltheway}.
By \cite{brauer2,olsson}, it follows that there are at most three isomorphism classes of simple
$B$-modules.
From Erdmann's classification of all blocks of tame representation type 
in \cite{erd}, it follows that the quiver and relations of the basic algebra
of $B$ can be given explicitly and that, up to Morita equivalence, there
are 24 families of blocks  $B$. 

Using \cite{eisele,erdsemid,erd,holm}, we now give a  description of these families as follows. 
By \cite[pp. 294--306]{erd},
there are 12 possible quivers $Q$ which can occur for basic algebras of dihedral, semidihedral
or quaternion type: $2\mathcal{A}$, $2\mathcal{B}$, $3\mathcal{A}$,
$3\mathcal{B}$, $3\mathcal{C}$, $3\mathcal{D}$, $3\mathcal{F}$, $3\mathcal{H}$, 
$3\mathcal{K}$, $3\mathcal{L}$, $3\mathcal{Q}$, $3\mathcal{R}$.
For each such quiver $Q$, we combine Erdmann's results in \cite{erdsemid,erd} and 
\cite[Prop. 4.2]{holm} with Eisele's results in
\cite{eisele} to provide the most accurate description of  
a full set of representatives of basic algebras $\Lambda=kQ/I$ for the Morita equivalence
classes of blocks $B$ as in Hypothesis \ref{hyp:alltheway} with Ext quiver $Q$.

We also provide the decomposition matrix for each block $B$. For better readability, all
decomposition matrices appear at the end of the paper.
As will be discussed in Section \ref{s:ordinarytame},
$B$ always contains exactly 4 ordinary irreducible characters of height 0 and, unless
$D$ is quaternion of order 8, exactly $2^{n-2}-1$ ordinary irreducible characters of height 1. 
If $D$ is quaternion of order 8, $B$ contains exactly 3 ordinary irreducible characters of height 1. If $B$ is generalized quaternion or semidihedral, there may be additional ordinary 
irreducible characters of height $n-2$. In the decomposition matrices, we list 
first the 4 ordinary irreducible characters of height 0, then the family of $2^{n-2}-1$ 
ordinary irreducible characters of height 1, and then the ordinary irreducible characters 
of height $n-2$ if they exist. Note that the family of $2^{n-2}-1$ characters of height 1
all define the same Brauer character on restricting to the 2-regular conjugacy classes 
of $G$. If $B$ has three isomorphism classes of simple modules, we moreover order the ordinary
irreducible characters according to the sign 
conventions described in \cite[Thm. 5]{brauer2} and \cite[Thms. 4.10, 4.11 and 4.15]{olsson}.

To distinguish between different defect groups, we use the notation
$\D(Q)$ (resp. $\SD(Q)$, resp. $\Q(Q)$) to mean that
$\Lambda=kQ/I$ is Morita equivalent to a block $B$ with dihedral (resp. semidihedral,
resp. generalized quaternion) defect groups. 

\subsection{Blocks with quiver $2\mathcal{A}$}
\label{ss:quiver2A}
$$\begin{array}{rl}
\raisebox{-2ex}{$2\mathcal{A}\quad=$}&\xymatrix @R=-.2pc {
0&1\\
\ar@(ul,dl)_{\alpha} \bullet \ar@<.8ex>[r]^{\beta} &\bullet\ar@<.9ex>[l]^{\gamma}}
\end{array}$$

By \cite{eisele} and \cite[p. 294]{erd}, if $B$ has dihedral defect groups and Ext quiver 
$2\mathcal{A}$, then
$B$ is Morita equivalent to $\D(2\mathcal{A})=k[2\mathcal{A}]/I_{\D(2\mathcal{A})}$  where
$$I_{\D(2\mathcal{A})}=\langle \beta\gamma, \alpha^2,
(\gamma\beta\alpha)^{2^{n-2}}-(\alpha\gamma\beta)^{2^{n-2}}\rangle .$$
The corresponding decomposition matrix is given in Figure \ref{fig:decompD(2A)}.

By \cite[p. 298]{erd}, if $B$ has semidihedral defect groups and Ext quiver $2\mathcal{A}$, then
there exists $c\in k$ such that
$B$ is Morita equivalent to either
$\SD(2\mathcal{A})_1(c)=k[2\mathcal{A}]/I_{\SD(2\mathcal{A})_1,c}$ or
$\SD(2\mathcal{A})_2(c)=k[2\mathcal{A}]/I_{\SD(2\mathcal{A})_2,c}$ where
$$\begin{array}{rl}
I_{\SD(2\mathcal{A})_1,c}\;=&\!\!\!\langle \alpha^2-c(\gamma\beta\alpha)^{2^{n-2}},
\beta\gamma\beta-\beta\alpha(\gamma\beta\alpha)^{2^{n-2}-1},
\gamma\beta\gamma-\alpha\gamma(\beta\alpha\gamma)^{2^{n-2}-1},\\
&  \alpha(\gamma\beta\alpha)^{2^{n-2}}\rangle,\\[1ex]
I_{\SD(2\mathcal{A})_2,c}\;=&\!\!\!\langle \beta\gamma,
\alpha^2-\gamma\beta(\alpha\gamma\beta)^{2^{n-2}-1}-c(\gamma\beta\alpha)^{2^{n-2}}, 
(\gamma\beta\alpha)^{2^{n-2}}-(\alpha\gamma\beta)^{2^{n-2}}\rangle.
\end{array}$$
The decomposition matrix for $\SD(2\mathcal{A})_1(c)$ is given in Figure \ref{fig:decompSD(2A)},
and the decomposition matrix for $\SD(2\mathcal{A})_2(c)$ is given in Figure 
\ref{fig:decompD(2A)}.

By \cite[p. 303]{erd}, if $B$ has generalized quaternion defect groups and Ext quiver 
$2\mathcal{A}$, then there exists $c\in k$ such that
$B$ is Morita equivalent to $\Q(2\mathcal{A})(c)=k[2\mathcal{A}]/I_{\Q(2\mathcal{A}),c}$ where
$$\begin{array}{rl}
I_{\Q(2\mathcal{A}),c}\;=&\!\!\!\langle 
\alpha^2-\gamma\beta(\alpha\gamma\beta)^{2^{n-2}-1}-c(\alpha\gamma\beta)^{2^{n-2}},
\beta\gamma\beta-\beta\alpha(\gamma\beta\alpha)^{2^{n-2}-1},\\
& \gamma\beta\gamma-\alpha\gamma(\beta\alpha\gamma)^{2^{n-2}-1},
\beta\alpha^2\rangle.
\end{array}$$
The decomposition matrix for $\Q(2\mathcal{A})(c)$ is given in Figure \ref{fig:decompSD(2A)}.

\subsection{Blocks with quiver $2\mathcal{B}$}
\label{ss:quiver2B}
$$\begin{array}{rl}
\raisebox{-2ex}{$2\mathcal{B}\quad=$}&\xymatrix @R=-.2pc {
0&1\\
\ar@(ul,dl)_{\alpha} \bullet \ar@<.8ex>[r]^{\beta} &\bullet\ar@<.9ex>[l]^{\gamma}
\ar@(ur,dr)^{\eta}}
\end{array}$$

By \cite{eisele} and \cite[p. 295]{erd}, if $B$ has dihedral defect groups and Ext quiver 
$2\mathcal{B}$, then
$B$ is Morita equivalent to $\D(2\mathcal{B})=k[2\mathcal{B}]/I_{\D(2\mathcal{B})}$  where
$$I_{\D(2\mathcal{B})}=\langle \eta\beta,\gamma\eta,\beta\gamma, \alpha^2,
\gamma\beta\alpha-\alpha\gamma\beta,\eta^{2^{n-2}}-\beta\alpha\gamma\rangle.$$
The corresponding decomposition matrix is given in Figure \ref{fig:decompD(2B)}.

By \cite[Lemmas (8.11) and (8.15)]{erdsemid}, \cite[p. 299]{erd} and \cite[Prop. 4.2]{holm},
if $B$ has semidihedral defect groups and Ext quiver $2\mathcal{B}$, then
there exists $c\in k$
such that
$B$ is Morita equivalent to either
$\SD(2\mathcal{B})_1(c)=k[2\mathcal{B}]/I_{\SD(2\mathcal{B})_1,c}$ or
$\SD(2\mathcal{B})_2(c)=k[2\mathcal{B}]/I_{\SD(2\mathcal{B})_2,c}$ or
$\SD(2\mathcal{B})_4(c)=k[2\mathcal{B}]/I_{\SD(2\mathcal{B})_4,c}$
where
$$\begin{array}{rl}
I_{\SD(2\mathcal{B})_1,c}\;=&\!\!\!\langle \eta\beta,\gamma\eta,\beta\gamma, 
\alpha^2-\gamma\beta-c\,\alpha\gamma\beta,
\gamma\beta\alpha-\alpha\gamma\beta,\eta^{2^{n-2}}-\beta\alpha\gamma
\rangle,\\[1ex]
I_{\SD(2\mathcal{B})_2,c}\;=&\!\!\!\langle \eta\beta-\beta\alpha(\gamma\beta\alpha),
\gamma\eta-\alpha\gamma(\beta\alpha\gamma),\alpha^2-c(\gamma\beta\alpha)^2,
\beta\gamma-\eta^{2^{n-2}-1},\\
& \eta^2\beta,\gamma\eta^2\rangle,\\[1ex]
I_{\SD(2\mathcal{B})_4,c}\;=&\!\!\!\langle \gamma\eta-\alpha\gamma,\beta\alpha-\eta\beta,
\alpha^{2^{n-2}+1},\eta^{2^{n-2}+1},\beta\alpha^{2^{n-2}-1},\alpha^{2^{n-2}-1}\gamma,\\
&\gamma\eta^{2^{n-2}-1}, \eta^{2^{n-2}-1}\beta,
\gamma\beta-\alpha^2,\beta\gamma-\eta^2(1+c\,\eta^{2^{n-2}-2})
\rangle.
\end{array}$$
The decomposition matrix for $\SD(2\mathcal{B})_1(c)$ is given in Figure \ref{fig:decompD(2B)},
the decomposition matrix for $\SD(2\mathcal{B})_2(c)$ is given in Figure \ref{fig:decompSD(2B)},
and
the decomposition matrix for $\SD(2\mathcal{B})_4(c)$ is given in Figure 
\ref{fig:decompSD(2B)weird}.

By \cite[IX.4.1 and pp. 303--304]{erd},
if $B$ has 
generalized quaternion defect groups and Ext quiver $2\mathcal{B}$, then
there exists $c\in k$, respectively $p(t)\in k[t]$ with $p(0)=1$ and $a,c\in k$ with $a\neq 0$,
such that $B$ is Morita equivalent to either 
$\Q(2\mathcal{B})_1(c)=k[2\mathcal{B}]/I_{\Q(2\mathcal{B})_1,c}$ 
or $\Q(2\mathcal{B})_2(p,a,c)=k[2\mathcal{B}]/I_{\Q(2\mathcal{B})_2,p,a,c}$ 
where
$$\begin{array}{rl}
I_{\Q(2\mathcal{B})_1,c}\;=&\!\!\!\langle 
\eta\beta-\beta\alpha(\gamma\beta\alpha),
\gamma\eta-\alpha\gamma(\beta\alpha\gamma),
\alpha^2-\gamma\beta(\alpha\gamma\beta)-c(\alpha\gamma\beta)^2,\\
&\beta\gamma-\eta^{2^{n-2}-1}, \beta\alpha^2\rangle,\\[1ex]
I_{\Q(2\mathcal{B})_2,p,a,c}\;=&\!\!\!\langle 
\gamma\eta-\alpha\gamma,\beta\alpha-\eta\beta,
\alpha^{2^{n-2}+1},\eta^{2^{n-2}+1},\beta\alpha^{2^{n-2}-1},\alpha^{2^{n-2}-1}\gamma,\\
& \gamma\beta-p(\alpha)\alpha^2,
\beta\gamma-p(\eta)\eta^2-a\eta^{2^{n-2}-1}-c\eta^{2^{n-2}}\rangle.
\end{array}$$
The decomposition matrix for $\Q(2\mathcal{B})_1(c)$ is given in Figure \ref{fig:decompSD(2B)},
and the decomposition matrix for $\Q(2\mathcal{B})_2(p,a,c)$ is given in Figure 
\ref{fig:decompSD(2B)weird}. Note that by \cite[Lemma 3.3]{olsson}, $\Q(2\mathcal{B})_2(p,a,c)$
can actually not occur as a block.

\subsection{Blocks with quiver $3\mathcal{A}$}
\label{ss:quiver3A}

$$
\xymatrix @R=-.2pc {
&0&\\
3\mathcal{A}= \quad 1\; \bullet \ar@<.8ex>[r]^(.66){\beta} \ar@<1ex>[r];[]^(.34){\gamma}
& \bullet \ar@<.8ex>[r]^(.44){\delta} \ar@<1ex>[r];[]^(.56){\eta} & \bullet\; 2}
$$

By \cite[IX.5.4 and p. 295]{erd}, if $B$ has dihedral defect groups and Ext quiver 
$3\mathcal{A}$, then
$B$ is Morita equivalent to $\D(3\mathcal{A})_1=k[3\mathcal{A}]/I_{\D(3\mathcal{A})_1}$  where
$$I_{\D(3\mathcal{A})_1}=\langle \gamma\beta,\delta\eta,
(\eta\delta\beta\gamma)^{2^{n-2}}-(\beta\gamma\eta\delta)^{2^{n-2}}\rangle.$$
The corresponding decomposition matrix is given in Figure \ref{fig:decompD(3A)}.

By \cite[IX.5.3 and pp. 299--300]{erd}, 
if $B$ has semidihedral defect groups and Ext quiver $3\mathcal{A}$, then
$B$ is Morita equivalent to $\SD(3\mathcal{A})_1=k[3\mathcal{A}]/I_{\SD(3\mathcal{A})_1}$  where
$$I_{\SD(3\mathcal{A})_1}=\langle \gamma\beta,
\delta\eta\delta-\delta\beta\gamma(\eta\delta\beta\gamma)^{2^{n-2}-1},
\eta\delta\eta-\beta\gamma\eta(\delta\beta\gamma\eta)^{2^{n-2}-1}\rangle.$$
The decomposition matrix for $\SD(3\mathcal{A})_1$ is given in Figure \ref{fig:decompSD(3A)}.

By \cite[IX.5.2 and pp. 304--305]{erd}, if $B$ has generalized quaternion defect groups and 
Ext quiver $3\mathcal{A}$, then
$B$ is Morita equivalent to $\Q(3\mathcal{A})_2=k[3\mathcal{A}]/I_{\Q(3\mathcal{A})_2}$  where
$$\begin{array}{rl}
I_{\Q(3\mathcal{A})_2}\;=&\!\!\!\langle 
\beta\gamma\beta-\eta\delta\beta(\gamma\eta\delta\beta)^{2^{n-2}-1},
\gamma\beta\gamma-\gamma\eta\delta(\beta\gamma\eta\delta)^{2^{n-2}-1},
\delta\beta\gamma\beta,\\
&\delta\eta\delta-\delta\beta\gamma(\eta\delta\beta\gamma)^{2^{n-2}-1},
\eta\delta\eta-\beta\gamma\eta(\delta\beta\gamma\eta)^{2^{n-2}-1},
\gamma\eta\delta\eta\rangle.
\end{array}$$
The decomposition matrix for $\Q(3\mathcal{A})_2$ is given in Figure \ref{fig:decompQ(3A)}.

\subsection{Blocks with quiver $3\mathcal{B}$}
\label{ss:quiver3B}

$$
\xymatrix @R=-.2pc {
&1&0&\\
 3\mathcal{B}= \quad&\ar@(ul,dl)_{\alpha} \bullet \ar@<.8ex>[r]^{\beta} \ar@<.9ex>[r];[]^{\gamma}
& \bullet \ar@<.8ex>[r]^(.46){\delta} \ar@<.9ex>[r];[]^(.54){\eta} & \bullet\;2}
$$

By \cite[IX.5.4 and pp. 295--296]{erd}, if $B$ has dihedral defect groups and Ext quiver 
$3\mathcal{B}$, then
$B$ is Morita equivalent to $\D(3\mathcal{B})_1=k[3\mathcal{B}]/I_{\D(3\mathcal{B})_1}$  where
$$I_{\D(3\mathcal{B})_1}=\langle \beta\alpha,\alpha\gamma, \gamma\beta,\delta\eta,
\eta\delta\beta\gamma-\beta\gamma\eta\delta,\alpha^{2^{n-2}}-\gamma\eta\delta\beta\rangle.$$
The corresponding decomposition matrix is given in Figure \ref{fig:decompD(3B)}.

By \cite[p. 300]{erd}, if $B$ has semidihedral defect groups and Ext quiver $3\mathcal{B}$, then
$B$ is Morita equivalent to either
$\SD(3\mathcal{B})_1=k[3\mathcal{B}]/I_{\SD(3\mathcal{B})_1}$ or
$\SD(3\mathcal{B})_2=k[3\mathcal{B}]/I_{\SD(3\mathcal{B})_2}$ where
$$\begin{array}{rl}
I_{\SD(3\mathcal{B})_1}\;=&\!\!\!\langle \beta\alpha,\alpha\gamma, \gamma\beta,
\delta\eta\delta-\delta\beta\gamma, \eta\delta\eta-\beta\gamma\eta,
\alpha^{2^{n-2}}-\gamma\eta\delta\beta\rangle,\\[1ex]
I_{\SD(3\mathcal{B})_2}\;=&\!\!\!\langle \delta\eta,\gamma\beta-\alpha^{2^{n-2}-1},
\alpha\gamma-\gamma\eta\delta(\beta\gamma\eta\delta),
\beta\alpha-\eta\delta\beta(\gamma\eta\delta\beta)\rangle.
\end{array}$$
The decomposition matrix for $\SD(3\mathcal{B})_1$ is given in Figure \ref{fig:decompSD(3B)1},
and
the decomposition matrix for $\SD(3\mathcal{B})_2$ is given in Figure \ref{fig:decompSD(3B)2}.

By \cite[p. 305]{erd}, if $B$ has generalized quaternion defect groups and 
Ext quiver $3\mathcal{B}$, then
$B$ is Morita equivalent to $\Q(3\mathcal{B})=k[3\mathcal{B}]/I_{\Q(3\mathcal{B})}$  where
$$\begin{array}{rl}
I_{\Q(3\mathcal{B})}\;=&\!\!\!\langle \gamma\beta-\alpha^{2^{n-2}-1},
\alpha\gamma-\gamma\eta\delta(\beta\gamma\eta\delta), 
\beta\alpha-\eta\delta\beta(\gamma\eta\delta\beta),
\delta\eta\delta-\delta\beta\gamma(\eta\delta\beta\gamma),\\
& 
\eta\delta\eta-\beta\gamma\eta(\delta\beta\gamma\eta), \beta\alpha^2,\delta\eta\delta\beta\rangle.
\end{array}$$
The decomposition matrix for $\Q(3\mathcal{B})$ is given in Figure \ref{fig:decompQ(3B)}.

\subsection{Blocks with quiver $3\mathcal{C}$}
\label{ss:quiver3C}

$$
\xymatrix @R=-.2pc {
&&0&\\
 3\mathcal{C}= \quad&1\; \bullet \ar@<.8ex>[r]^{\beta} \ar@<.9ex>[r];[]^{\gamma}
& \bullet \ar@(dl,dr)_{\rho}\ar@<.8ex>[r]^(.46){\delta} \ar@<.9ex>[r];[]^(.54){\eta} & \bullet\;2}
$$

By \cite[VI.5]{erd} (resp. by \cite[IX.5.3 and p. 305]{erd}),
there are no blocks $B$ with dihedral (resp.
generalized quaternion) defect groups that have Ext quiver $3\mathcal{C}$.

By \cite[IX.5.3 and pp. 300--301]{erd}, if $B$ has semidihedral defect groups and Ext quiver $3\mathcal{C}$, then
$B$ is Morita equivalent to either
$\SD(3\mathcal{C})_{2,1}=k[3\mathcal{C}]/I_{\SD(3\mathcal{C})_{2,1}}$ or
$\SD(3\mathcal{C})_{2,2}=k[3\mathcal{C}]/I_{\SD(3\mathcal{C})_{2,2}}$ where
$$\begin{array}{rl}
I_{\SD(3\mathcal{C})_{2,1}}\;=&\!\!\!\langle \rho\beta,\delta\rho,\rho\eta,\gamma\rho,
\beta\gamma-\eta\delta,(\beta\gamma)^2-\rho^{2^{n-2}},
\delta\beta\gamma\beta,\gamma\eta\delta\eta\rangle,\\[1ex]
I_{\SD(3\mathcal{C})_{2,2}}\;=&\!\!\!\langle \rho\beta,\delta\rho,\rho\eta,\gamma\rho,
\beta\gamma-\eta\delta,(\beta\gamma)^{2^{n-2}}-\rho^2,
\delta\beta(\gamma\beta)^{2^{n-2}-1},\\
&\gamma\eta(\delta\eta)^{2^{n-2}-1} \rangle.
\end{array}$$
The decomposition matrix for $\SD(3\mathcal{C})_{2,1}$ is given in Figure 
\ref{fig:decompSD(3C)1}, and
the decomposition matrix for $\SD(3\mathcal{C})_{2,2}$ is given in Figure 
\ref{fig:decompSD(3C)2}.

\subsection{Blocks with quiver $3\mathcal{D}$}
\label{ss:quiver3D}

$$
\xymatrix @R=-.2pc {
&1&0&2\\
 3\mathcal{D}= \quad&\ar@(ul,dl)_{\alpha} \bullet \ar@<.8ex>[r]^{\beta} \ar@<.9ex>[r];[]^{\gamma}
& \bullet \ar@<.8ex>[r]^(.46){\delta} \ar@<.9ex>[r];[]^(.54){\eta} & \bullet \ar@(ur,dr)^{\xi}}
$$

By \cite[IX.5.1, IX.5.4 and p. 296]{erd} (resp. by \cite[IX.5.1 and p. 306]{erd}),
there are no blocks $B$ with dihedral (resp.
generalized quaternion) defect groups that have Ext quiver $3\mathcal{D}$.

By \cite[p. 301]{erd}, if $B$ has semidihedral defect groups and Ext quiver $3\mathcal{D}$, then
$B$ is Morita equivalent to $\SD(3\mathcal{D})=k[3\mathcal{D}]/I_{\SD(3\mathcal{D})}$  where
$$I_{\SD(3\mathcal{D})}=\langle
\xi\delta,\eta\xi,\delta\eta,\gamma\beta-\alpha^{2^{n-2}-1},\alpha\gamma-\gamma\eta\delta,
\beta\alpha-\eta\delta\beta,\xi^2-\delta\beta\gamma\eta\rangle.$$
The decomposition matrix for $\SD(3\mathcal{D})$ is given in Figure \ref{fig:decompSD(3B)1}.

\subsection{Blocks with quiver $3\mathcal{F}$}
\label{ss:quiver3F}

By \cite[VI.5]{erd} (resp. by \cite[VII.4]{erd}), there are no blocks $B$ with dihedral (resp.
generalized quaternion) defect groups that have Ext quiver $3\mathcal{F}$.
By \cite[IX.5.2 and p. 301]{erd}, there are also no blocks $B$ with semidihedral 
defect groups that have Ext quiver $3\mathcal{F}$.

\subsection{Blocks with quiver $3\mathcal{H}$}
\label{ss:quiver3H}

$$3\mathcal{H}=\vcenter{\xymatrix @R1.8pc  {
 0\,\bullet \ar@<.7ex>[rr]^{\beta} \ar@<.8ex>[rr];[]^{\gamma} \ar@<.8ex>[rdd];[]^{\lambda}
&&\bullet\ar@<.7ex>[ldd]^{\delta} \ar@<.8ex>[ldd];[]^{\eta}\,1\\&&\\ &
\unter{\mbox{\normalsize $\bullet$}}{\mbox{\normalsize $2$}}& }}$$

By \cite[VI.5]{erd} (resp. by \cite[VII.4]{erd}), there are no blocks $B$ with dihedral (resp.
generalized quaternion) defect groups that have Ext quiver $3\mathcal{H}$.

By \cite[p. 301]{erd}, if $B$ has semidihedral defect groups and Ext quiver $3\mathcal{H}$, then
$B$ is Morita equivalent to $\SD(3\mathcal{H})_1=k[3\mathcal{H}]/I_{\SD(3\mathcal{H})_1}$
or $\SD(3\mathcal{H})_2=k[3\mathcal{H}]/I_{\SD(3\mathcal{H})_2}$  where
$$\begin{array}{rl}
I_{\SD(3\mathcal{H})_1}\;=&\!\!\!\langle
\lambda\delta-\gamma\beta\gamma,\beta\lambda-\eta(\delta\eta)^{2^{n-2}-1},
\eta\delta\beta,\delta\beta\gamma,\gamma\eta\rangle,\\[1ex]
I_{\SD(3\mathcal{H})_2}\;=&\!\!\!\langle
\lambda\delta-\gamma(\beta\gamma)^{2^{n-2}-1},\beta\lambda-\eta\delta\eta,
\eta\delta\beta,\delta\beta\gamma,\gamma\eta\rangle.
\end{array}$$
The decomposition matrix for $\SD(3\mathcal{H})_1$ is given in Figure \ref{fig:decompSD(3H)1},
and the decomposition matrix for $\SD(3\mathcal{H})_2$ is given in Figure 
\ref{fig:decompSD(3H)2}.

\subsection{Blocks with quiver $3\mathcal{K}$}
\label{ss:quiver3K}

$$3\mathcal{K}=\vcenter{\xymatrix @R1.8pc {
 0\,\bullet \ar@<.7ex>[rr]^{\beta} \ar@<.8ex>[rr];[]^{\gamma}\ar@<.7ex>[rdd]^{\kappa} \ar@<.8ex>[rdd];[]^{\lambda}
&&\bullet\ar@<.7ex>[ldd]^{\delta} \ar@<.8ex>[ldd];[]^{\eta}\,1\\&&\\ &
\unter{\mbox{\normalsize $\bullet$}}{\mbox{\normalsize $2$}}& }}$$

By \cite[IX.5.2 and p. 302]{erd}, there are no blocks $B$ with semidihedral 
defect groups that have Ext quiver $3\mathcal{K}$.

By \cite[p. 296]{erd}, if $B$ has dihedral defect groups and Ext quiver $3\mathcal{K}$, then
$B$ is Morita equivalent to $\D(3\mathcal{K})=k[3\mathcal{K}]/I_{\D(3\mathcal{K})}$  where
$$I_{\D(3\mathcal{K})}=\langle \delta\beta,\lambda\delta,\beta\lambda,\kappa\gamma,\eta\kappa,
\gamma\eta, \gamma\beta-\lambda\kappa,\kappa\lambda-(\delta\eta)^{2^{n-2}},
(\eta\delta)^{2^{n-2}}-\beta\gamma\rangle.$$
The corresponding decomposition matrix is given in Figure \ref{fig:decompD(3K)}.

By \cite[p. 306]{erd}, if $B$ has generalized quaternion defect groups and 
Ext quiver $3\mathcal{K}$, then
$B$ is Morita equivalent to $\Q(3\mathcal{K})=k[3\mathcal{K}]/I_{\Q(3\mathcal{K})}$  where
$$\begin{array}{rl}
I_{\Q(3\mathcal{K})}\;=&\!\!\!\langle \delta\beta-\kappa\lambda\kappa,
\gamma\eta-\lambda\kappa\lambda,
\lambda\delta-\gamma\beta\gamma,
\eta\kappa-\beta\gamma\beta,\beta\lambda-\eta(\delta\eta)^{2^{n-2}-1},\\
&\kappa\gamma-\delta(\eta\delta)^{2^{n-2}-1},
\delta\beta\gamma, \gamma\eta\delta, \eta\kappa\lambda
\rangle.
\end{array}$$
The decomposition matrix for $\Q(3\mathcal{K})$ is given in Figure \ref{fig:decompQ(3K)}.

\subsection{Blocks with quivers $3\mathcal{L}$, $3\mathcal{Q}$ or $3\mathcal{R}$}
\label{ss:quiver3LQR}

By \cite[VIII.2.0]{erd} (resp. by \cite[VII.4]{erd}), there are no blocks $B$ with semidihedral 
(resp. generalized quaternion) defect groups that have Ext quivers $3\mathcal{L}$, 
$3\mathcal{Q}$ or $3\mathcal{R}$.
By \cite[IX.5.4]{erd}, there are also no blocks $B$ with dihedral 
defect groups that have Ext quivers $3\mathcal{L}$, $3\mathcal{Q}$ or $3\mathcal{R}$.


\section{Ordinary characters belonging to tame blocks}
\label{s:ordinarytame}

Assume Hypothesis \ref{hyp:alltheway}. In particular,
$p=2$, $n\ge 2$, and $D$  is dihedral, semidihedral or generalized quaternion. 

For $2\le\ell\le n-1$, define $\zeta_\ell=\zeta^{2^{n-1-\ell}}$, so that $\zeta_\ell$
is a primitive $2^\ell$-th root of unity. 
It follows from \cite[Sect. VII]{brIV}, \cite{brauer2,olsson} and \cite{fong} (see below) that every 
ordinary irreducible 
character of $G$ which belongs to $B$ is realizable over $F(\zeta)$, i.e. it corresponds to
an absolutely irreducible $F(\zeta)G$-module.
In particular, if we are only interested in the ordinary irreducible characters of $G$ belonging to
$B$, we can replace $\xi$ in the paragraph before Remark \ref{rem:fusion} by $\zeta$.

In \cite{brauer2} (resp. \cite{olsson}), the ordinary irreducible characters of $G$ belonging to $B$
were analyzed if $n\ge 3$ and $D$ is dihedral (resp. semidihedral or generalized quaternion).
In the notation of \cite[Sect. 4]{brauer2} (resp. \cite[Sect. 2]{olsson}), this means that we 
are either in Case $(aa)$ or in Case $(ab)$ or $(ba)$, see \cite[Thm. 2]{brauer2} 
(resp. \cite[Thms. 3.14-3.17]{olsson}). Note that Case $(ba)$ can only occur when
$D$ is semidihedral.
In particular, in Case $(aa)$ (resp. Case $(ab)$ or $(ba)$) there are precisely three
(resp. two) isomorphism classes of simple $B$-modules.

\begin{rem}
\label{rem:cases}
\begin{enumerate}
\item[(a)]
If $D$ is dihedral of order 4, i.e. $n=2$, it follows from \cite{erdklein} that $B$ is Morita
equivalent to either $\D(3\mathcal{A})_1$ or $\D(3\mathcal{K})$. Moreover, there are
precisely 4 ordinary irreducible characters belonging to $B$ and they all have height 0.
In the decomposition matrices in Figures \ref{fig:decompD(3A)} and \ref{fig:decompD(3K)},
these characters are $\chi_1,\chi_2,\chi_3,\chi_4$. 
By \cite[Prop. (7D)]{brIV} and \cite{fong}, it follows that all these characters are
realizable over $F$, i.e. they correspond to absolutely irreducible $FG$-modules.

\item[(b)]
If $D$ is quaternion of order 8, i.e. $n=3$, it follows from \cite[p. 220 and Thm. 3.17]{olsson} that
we are in Case $(aa)$ and that there are precisely 4 (resp. 3) ordinary irreducible characters 
of height 0 (resp. $1=n-2$) belonging to $B$. Moreover, $B$ is Morita equivalent
to one of the algebras in  $\{\Q(3\mathcal{A})_2,\Q(3\mathcal{B}),\Q(3\mathcal{K})\}$ and in 
the decomposition matrices in Figures \ref{fig:decompQ(3A)}, \ref{fig:decompQ(3B)} and
\ref{fig:decompQ(3K)}, the ordinary irreducible characters of height 0 (resp. 1) are 
$$\chi_1,\chi_2,\chi_3,\chi_4\qquad \mbox{(resp. $\chi_{5,1},\chi_6,\chi_7$).}$$
By \cite[Prop. 4.2]{olsson} and \cite{fong}, all these characters are
realizable over $F$, i.e. they correspond to absolutely irreducible $FG$-modules.

\item[(c)]
If $D$ is not quaternion of order 8, then there are precisely 4 (resp. $2^{n-2}-1$)
ordinary irreducible characters of height 0 (resp. 1) belonging to $B$. 
If $D$ is dihedral, these are all ordinary irreducible characters belonging to $B$.
If $D$ is semidihedral, there is 0 (resp. 1) additional ordinary irreducible character
of height $n-2$ belonging to $B$ if we are in Case $(ab)$ (resp. Cases $(ba)$ or $(aa)$).
If $D$ is generalized quaternion of order $\ge 16$, there are 
1 (resp. 2) additional ordinary irreducible characters of height $n-2$ belonging to $B$
if we are in Case $(ab)$ (resp. Case $(aa)$). 
In the decomposition matrices in Figures \ref{fig:decompD(2A)}--\ref{fig:decompQ(3K)},
the ordinary irreducible characters of height 0 (resp. 1) are 
$$\chi_1,\chi_2,\chi_3,\chi_4\qquad \mbox{(resp. $\chi_{5,i}$ for $1\le i\le 2^{n-2}-1$), }$$
whereas  the ordinary irreducible characters of
height $n-2$ are $\chi_6$ or $\chi_6,\chi_7$, provided they exist.
\end{enumerate}
\end{rem}

Let $n\ge 3$, and let
$\sigma$ be an element of  order $2^{n-1}$ in $D$.
By \cite{brauer2,olsson}, there exists a block $b_\sigma$ of $kC_G(\sigma)$ with 
$b_\sigma^G=B$ which contains a unique irreducible Brauer character $\varphi^{(\sigma)}$ such that 
the following is true. There is an ordering of $(1,2,\ldots,2^{n-2}-1)$ such that for 
$1\le i\le 2^{n-2}-1$, the generalized decomposition number of $G$ corresponding to
$\sigma$, $\chi_{5,i}$ and $\varphi^{(\sigma)}$ has the form
\begin{equation}
\label{eq:great0}
d_{\chi_{5,i},\,\varphi^{(\sigma)}}^{\,(\sigma)} =
\left\{\begin{array}{ll}
\zeta^{i}+\zeta^{-i}&\mbox{ if $D$ is dihedral or generalized quaternion,}\\
&\mbox{ or if $D$ is semidihedral and $i$ is even,}\\[.2cm]
\zeta^{i}-\zeta^{-i}&\mbox{ if $D$ is semidihedral and $i$ is odd with $i\le2^{n-3}-1$,}
\\[.2cm]
\zeta^{-i}-\zeta^i&\mbox{ if $D$ is semidihedral and $i$ is odd with $i\ge 2^{n-3}+1$.}
\end{array}\right.
\end{equation}
Note that the formulas in (\ref{eq:great0})
for $D$ semidihedral and $i$ odd follow since the $D$-conjugacy classes
of elements of order $2^{n-1}$ in $D$ are represented by
$$\sigma^1, \sigma^3,\ldots,\sigma^{2^{n-3}-1},\quad \sigma^{-(2^{n-3}+1)},\sigma^{-(2^{n-3}+3)},
\ldots,\sigma^{-(2^{n-2}-1)}.$$
If $D$ is quaternion of order 8, we have that 
$d_{\chi,\varphi^{(\sigma)}}^{\,(\sigma)}=\zeta+\zeta^{-1}=0$
for $\chi\in\{\chi_{5,1},\chi_6,\chi_7\}$.

For $2\le \ell\le n-2$ define $\nu_\ell =\zeta_{\ell}+\zeta_{\ell}^{-1}$, and define
\begin{equation}
\label{eq:nu}
\nu_{n-1} =\left\{\begin{array}{ll}
\zeta+\zeta^{-1}&\mbox{  if $D$ is dihedral or generalized quaternion,}\\[.2cm]
\zeta-\zeta^{-1}&\mbox{ if $D$ is semidihedral.} \end{array}\right.
\end{equation}
Note that $W$ contains all roots of unity of order not divisible by $2$. Hence by 
\cite{brauer2,olsson} and by \cite{fong}, the ordinary irreducible characters of height 0 or $n-2$
belonging to $B$ correspond to simple $FG$-modules. On the other hand, the characters 
$\chi_{5,i}$, $i=1,\ldots,2^{n-2}-1$,
fall into $n-2$ Galois orbits $\mathcal{O}_2,\ldots, \mathcal{O}_{n-1}$ under the action of 
$\mathrm{Gal}(F(\nu_{n-1})/F)$. Namely for $2\le\ell\le n-1$, 
$$\mathcal{O}_{\ell}=\{ \chi_{5,2^{n-1-\ell}(2u-1)} \;|\; 1\le u\le 2^{\ell-2}\}.$$
The field generated by the character values of each $\xi_\ell\in\mathcal{O}_\ell$ over $F$ is 
$F(\nu_\ell)$. 
Hence by \cite{fong}, each $\xi_\ell$ corresponds to an absolutely irreducible 
$F(\nu_\ell) G$-module $X_\ell$. 
By \cite[Satz V.14.9]{hup},  this implies that for $2\le\ell\le n-1$, the Schur index of each 
$\xi_\ell\in\mathcal{O}_\ell$ over $F$ is $1$. Hence we obtain $n-2$ non-isomorphic simple 
$FG$-modules $V_2,\ldots,V_{n-1}$
with characters $\rho_2,\ldots, \rho_{n-1}$ satisfying
\begin{equation}
\label{eq:goodchar1}
\rho_\ell =\sum_{\xi_\ell\in\mathcal{O}_\ell}\xi_\ell = 
\sum_{u=1}^{2^{\ell-2}} \chi_{5,2^{n-1-\ell}(2u-1)} 
\qquad\mbox{for $2\le \ell \le n-1$.}
\end{equation}
By \cite[Hilfssatz V.14.7]{hup}, $\mathrm{End}_{FG}(V_\ell)$ is a commutative $F$-algebra isomorphic 
to the field generated over $F$ by the character values of any $\xi_\ell\in\mathcal{O}_\ell$. 
This means
\begin{equation}
\label{eq:goodendos}
\mathrm{End}_{FG}(V_\ell)\cong F(\nu_\ell)\qquad\mbox{for $2\le \ell \le n-1$.}
\end{equation}

Suppose $v_1,\ldots,v_{l(\sigma)}$ form a complete system of representatives of 
$C_G(\sigma)$-con\-ju\-gacy classes of $2$-regular elements in $C_G(\sigma)$ with $v_1=1_G$. 
By (\ref{eq:brauer4}),
for all $1\le i\le 2^{n-2}-1$, the generalized decomposition number of $G$ corresponding to
$\chi_{5,i}$, $\sigma$ and $\varphi^{(\sigma)}$ can be written as a $W$-linear
combination of $\chi_{5,i}(\sigma v_1),\ldots,\chi_{5,i}(\sigma v_{l(\sigma)})$,
say
\begin{equation}
\label{eq:duhh1}
d_{\chi_{5,i},\,\varphi^{(\sigma)}}^{\,(\sigma)} =
\tilde{w}_1\cdot \chi_{5,i}(\sigma v_1)+\cdots+ \tilde{w}_{l(\sigma)}\cdot
\chi_{5,i}(\sigma v_{l(\sigma)})
\end{equation}
for certain $\tilde{w}_1,\ldots,\tilde{w}_{l(\sigma)}\in W$.

By \cite[Thm. 5]{brauer2} and \cite[Prop. 4.6]{olsson}, the characters $\chi_{5,i}$ have the 
same degree $x$ and they are all of height 1 for $1\le i\le 2^{n-2}-1$.
Hence $x=2^{a-n+1}x^*$ where $\#G=2^a\cdot g^*$ and $x^*$ and $g^*$ are odd. Since the 
centralizer $C_G(\sigma)$ contains $\langle \sigma \rangle$, we have $\# C_G(\sigma)=
2^{n-1}\cdot 2^b\cdot m^*$ where $b\ge 0$ and $m^*$ is odd. 

For $1\le j\le l(\sigma)$, let $C_j$ be the conjugacy class in $G$ of $\sigma v_j$, and let 
$t(C_j)\in WG$ be the class sum of $C_j$. Since for all $j$,  the centralizer 
$C_G(\sigma v_j)$ contains $\langle \sigma \rangle$, we have $\# C_G(\sigma v_j)=
2^{n-1}\cdot 2^{b_j}\cdot m_j^*$ where $b_j\ge 0$ and $m_j^*$ is odd. 
We want to determine the action of $t(C_j)$ on $V_\ell$ for $2\le \ell\le n-1$. For this, we identify 
$\mathrm{End}_{FG}(V_\ell)\cong F(\nu_\ell)$ with 
$\mathrm{End}_{F(\nu_\ell) G}(X_\ell)$ for one particular absolutely irreducible 
$F(\nu_\ell) G$-constituent $X_\ell$ of $V_\ell$ with character $\xi_\ell$. 
Using (\ref{eq:goodchar1}), we choose $\xi_\ell=\chi_{5,2^{n-1-\ell}}$. Then
for $2\le \ell \le n-1$, the action of $t(C_j)$ on $V_\ell$ is given as multiplication by 
$\mu_j(\ell)$, where
\begin{equation}
\label{eq:goodone}
\mu_j(\ell)=\frac{\# C_j}{\xi_\ell(1)} \cdot \xi_\ell(\sigma v_j)
=2^{-b_j} \frac{g^*}{m_j^*\cdot x^*}\cdot \xi_\ell(\sigma v_j).
\end{equation}
For $1\le j\le l(\sigma)$, define $w_j=2^{b_j}\frac{m_j^*\cdot x^*}{g^*} \tilde{w}_j$.
Then $w_j\in W$ since $g^*$ is odd, and $w_j$ does not depend on $\ell$. 
By (\ref{eq:great0}) and (\ref{eq:nu}),
$$
d_{\xi_\ell,\,\varphi^{(\sigma)}}^{\,(\sigma)} \;=\;
d_{\chi_{5,2^{n-1-\ell}},\,\varphi^{(\sigma)}}^{\,(\sigma)} \;=\;
\nu_\ell.$$
Therefore, (\ref{eq:duhh1}) and (\ref{eq:goodone}) imply that 
\begin{equation}
\label{eq:duhh2}
\nu_\ell = 
w_1\cdot \mu_1(\ell) + \cdots +w_{l(\sigma)}\cdot \mu_{l(\sigma)}(\ell)
\end{equation}
where $w_1,\ldots,w_{l(\sigma)}\in W$ are independent of $\ell\in\{2,\ldots,n-1\}$. 

\begin{rem}
\label{rem:principal}
If $B$ is a principal block, the above formulas simplify considerably due to the fact that there
is very little fusion of $D$-conjugacy classes in $G$ in this case. More precisely, following Brauer's
arguments in \cite[Sect. VII]{brIII}, suppose $D$ is a Sylow 2-subgroup of $G$ and 
let $S=\langle \sigma \rangle$. If $\sigma^\lambda$ is not in the center of $D$,
then $S$ is a Sylow 2-subgroup of $C_G(\sigma^\lambda)$. Hence if $\sigma^\lambda$ and
$\sigma^\mu$ are conjugate in $G$, we can use Sylow's theorems to see that they are
conjugate in $N_G(S)$. Since $N_G(S)/C_G(S)$ is a 2-group, it must be of order 2.
Thus $N_G(S)$ is generated by $D$ and $C_G(S)$, which implies that 
$\sigma^\lambda$ and $\sigma^\mu$ are conjugate in $D$. Using 
(\ref{eq:fusion!}) together with \cite[Prop. (4A)]{brauer2} and \cite[Prop. 2.10]{olsson},
we obtain that 
\begin{equation}
\label{eq:principal}
\chi_{5,i}(\sigma) = 
d_{\chi_{5,i},\,\varphi^{(\sigma)}}^{\,(\sigma)} \;\varphi^{(\sigma)}(1_G)
\end{equation}
where $\varphi^{(\sigma)}$ is as in (\ref{eq:great0}).
Moreover, since $\varphi^{(\sigma)}$ is the unique Brauer character belonging to the block
$b_\sigma$ of $kC_G(\sigma)$ satisfying $b_\sigma^G=B$, it follows by Brauer's Third Main
Theorem (see e.g. \cite[Thm. 16.1]{alp}) that $b_\sigma$ is the principal block of $kC_G(\sigma)$,
which implies that $\varphi^{(\sigma)}$ is the trivial character of $b_\sigma$.
Putting (\ref{eq:principal}) into (\ref{eq:goodone}) for $j=1$ therefore implies that 
if $\omega=2^{b_1}\frac{m_1^*\cdot x^*}{g^*}$, which lies in $W$, then for all $2\le \ell\le n-1$
$$\nu_\ell = 
d_{\chi_{5,2^{n-1-\ell}},\,\varphi^{(\sigma)}}^{\,(\sigma)}
=\omega\cdot \mu_1(\ell).$$
Note that in \cite[Sect. 3.4]{3sim}, \cite[Sect. 5]{3quat} and \cite[Sect. 3.2]{2sim}, it was assumed that the
formula (\ref{eq:principal}) was also true for non-principal blocks. However, this may not be true
since there could be more fusion of $D$-conjugacy classes in $G$ in this case. 
More precisely, let $Y$ be a full set of representatives of $D$-conjugacy classes  
of the elements of the set $\{\sigma^r\;|\; r \mbox{ odd}\}$,
and let $Y_\sigma$ be the set of all $y\in Y$ which are 
conjugate to $\sigma$ in $G$. 
If $B$ is not principal, then $|Y_\sigma|$ may be strictly greater than 1.
By \cite[Prop. (4A)]{brauer2} and \cite[Prop. 2.10]{olsson}, the set
$\{(y,b_\sigma)\;|\; y\in Y_\sigma\}$ is a system of representatives for the conjugacy classes
of subsections for $B$ such that $y$ is conjugate to $\sigma$ in $G$.
Hence it follows from Remark \ref{rem:fusion} that for non-principal blocks $B$, 
(\ref{eq:principal}) has to be replaced by the formula
$$\chi_{5,i}(\sigma) = \sum_{y\in Y_\sigma}
d_{\chi_{5,i},\,\varphi^{(\sigma)}}^{\,(y)}\;\varphi^{(\sigma)}(1_G).$$
\end{rem}

\begin{dfn}
\label{def:seemtoneed}
Use the notation introduced above, and in particular (\ref{eq:nu}). Assume $n\ge 3$.
\begin{enumerate}
\item[(i)] 
Define
$$q_n(t)=\prod_{\ell=2}^{n-1} \mathrm{min.pol.}_F(\nu_\ell)$$
and let $R'=W[[t]]/(q_n(t))$.
\item[(ii)] Let $Z=\langle \sigma\rangle$, so that $Z$ is a cyclic group of order $2^{n-1}>2$.
Let $\tau:Z\to Z$ be the group automorphism which sends $\sigma$ to $\sigma^{-1}$ if $D$ is
dihedral or generalized quaternion, and which sends $\sigma$ to $\sigma^{-1+2^{n-2}}$ if $D$ is
semidihedral. Then $\tau$ can be extended to a $W$-algebra automorphism of the group ring $WZ$ which will again be denoted by $\tau$. Let $T(\sigma^2)=1+\sigma^2+\sigma^4+\cdots +\sigma^{2^{n-1}-2}$, and define 
$$S'= (WZ)^{\langle \tau\rangle}/\left(T(\sigma^2),\sigma T(\sigma^2)\right).$$
\end{enumerate}
\end{dfn}

\begin{rem}
\label{rem:ohyeah}
The minimal polynomial $\mathrm{min.pol.}_F(\nu_\ell)$ for $2\le\ell\le n-1$ is as follows:
\begin{eqnarray*}
\mathrm{min.pol.}_F(\nu_2)(t)&=&t,\\ 
\mathrm{min.pol.}_F(\nu_{\ell})(t)
&=&\mathrm{min.pol.}_F(\nu_{\ell-1})(t^2-2) \qquad \mbox{for $3\le\ell\le n-2$},\\
\mathrm{min.pol.}_F(\nu_{n-1})(t)
&=&\left\{\begin{array}{ll}
\mathrm{min.pol.}_F(\nu_{n-2})(t^2-2)&\mbox{ if $D$ is dihedral or}\\
&\mbox{ generalized quaternion,}\\[.2cm]
\mathrm{min.pol.}_F(\nu_{n-2})(t^2+2)&\mbox{ if $D$ is semidihedral.}
\end{array}\right.
\end{eqnarray*}
The $W$-algebra $R'$ from Definition \ref{def:seemtoneed} is a complete local 
commutative Noetherian ring with residue field $k$. Moreover, 
\begin{eqnarray*}
F\otimes_W R'&\cong& \prod_{\ell=2}^{n-1} F(\nu_\ell)\quad\mbox{ as $F$-algebras,}\\
k\otimes_W R'&\cong& k[t]/(t^{2^{n-2}-1})\quad\mbox{ as $k$-algebras.}
\end{eqnarray*}
Additionally, 
$R'$ is isomorphic to the $W$-subalgebra of
$\prod_{\ell=2}^{n-1} W[\nu_\ell]$
generated by the element 
$(\nu_\ell)_{\ell=2}^{n-1}$.
\end{rem}

\begin{lemma}
\label{lem:semidihedraltrick}
Using the notation of Definition $\ref{def:seemtoneed}$, there 
is a continuous $W$-algebra isomorphism $h:R'\to S'$ with $h(t)=\sigma+\tau(\sigma)$. 
In particular, $R'$ is isomorphic to a subquotient algebra of the group algebra $WD$.

Moreover, if $D$ is dihedral or semidihedral, then the ring $W[[t]]/(t\,q_n(t),2\,q_n(t))$
is also isomorphic to a subquotient algebra of the group algebra $WD$.
\end{lemma}

\begin{proof}
If $D$ is dihedral or generalized quaternion, this follows from \cite[Lemma 2.3.6]{3sim} and 
\cite[Lemma 5.3]{3quat}.

For the remainder of the proof, assume that $D$ is semidihedral (in particular, $n\ge 4$). 
Let $J=\sigma^{2^{n-2}}$ so that $\tau(\sigma)=J\sigma^{-1}$.
Note that $(WZ)^{\langle \tau\rangle}$ is generated as a $W$-algebra
by $(\sigma+J\sigma^{-1})$ and $J$. Moreover, $(WZ)^{\langle \tau\rangle}$ is a free 
$W$-module of rank $2^{n-2}+1$ with $W$-basis given by
\begin{equation}
\label{eq:basiss}
\begin{array}{c}
(\sigma^{\pm 1} + J\sigma^{\mp 1}),(\sigma^{\pm 3}+J\sigma^{\mp 3}),\ldots,
(\sigma^{\pm (2^{n-3}-1)}+J\sigma^{\mp(2^{n-3}-1)}),\\
1,J, (\sigma^2+\sigma^{-2}),(\sigma^4+\sigma^{-4}),\ldots, 
(\sigma^{2^{n-2}-2}+\sigma^{-(2^{n-2}-2)}).
\end{array}
\end{equation}
Note that the $W$-sugbalgebra of $(WZ)^{\langle \tau\rangle}$ generated by 
$(\sigma+J\sigma^{-1})$ is a free $W$-module of the same rank $2^{n-2}+1$ and with
almost the same $W$-basis as in (\ref{eq:basiss}) except that $J$ must be replaced by $2J$. 
It follows that $S'= (WZ)^{\langle \tau\rangle}/\left(T(\sigma^2),\sigma T(\sigma^2)\right)$
is generated as a $W$-algebra by the image of $(\sigma+J\sigma^{-1})$ in $S'$. 
Hence we have a surjective $W$-algebra homomorphism
$$f:\quad W[[t]] \to S'=(WZ)^{\langle \tau\rangle}/\left(T(\sigma^2),\sigma T(\sigma^2)\right)$$
sending $t$ to the image in $S'$ of $\sigma+J\sigma^{-1}=\sigma+\tau(\sigma)$. 
Using the injective $W$-algebra homomorphism 
\begin{eqnarray*}
\iota: \qquad WZ &\to& W\times W \times \prod_{\ell=2}^{n-1} W[\zeta_\ell]\\
\sigma &\mapsto& \left(1\;,\;-1\;\;\;,\;\left(\zeta_\ell\right)_{\ell=2}^{n-1}\;\right)
\end{eqnarray*}
it is straightforward to prove that there exists an odd integer $c_n$ such that
$q_n(\sigma+J\sigma^{-1})=c_n\,\sigma\,T(\sigma^2)$. More precisely, $c_4=3$
and $c_n=2\, c_{n-1}^2-1$ for $n\ge 5$.
Thus $q_n(t)$  lies in the kernel of $f$. Since
both $R'=W[[t]]/(q_n(t))$ and $S'$ are free as $W$-modules of the same rank $2^{n-2}-1$, 
it follows that $R'=W[[t]]/(q_n(t))\cong S'$ as $W$-algebras.

To finish the proof of Lemma \ref{lem:semidihedraltrick},
it suffices to show that the ring $W[[t]]/((t-2)\,q_n(t))$ is isomorphic to a subquotient algebra of 
$WD$. Define
$$\Theta=(WZ)^{\langle\tau\rangle}/\left(T(\sigma^2)- \sigma T(\sigma^2)\right).$$
Then $\Theta$ is isomorphic to a subquotient algebra of $WD$ and it is generated as a 
$W$-algebra by the image of $(\sigma+J\sigma^{-1})$ in $\Theta$. Moreover, $\Theta$ is a free 
$W$-module of rank $2^{n-2}$, since the ideal $\left(T(\sigma^2)- \sigma T(\sigma^2)\right)$ is 
generated over $W$ by $T(\sigma^2)- \sigma T(\sigma^2)$. 
Define a surjective $W$-algebra homomorphism
$\theta:W[[t]]\to \Theta$ by sending $t$ to the image in $\Theta$ of $\sigma+J\sigma^{-1}
=\sigma+\tau(\sigma)$.  
Using the above calculations, we see that
\begin{eqnarray*}
\theta((t-2)\,q_n(t)) &=& ((\sigma+J\sigma^{-1})-2)\,q_n(\sigma+J\sigma^{-1}) \\
&=& ((\sigma+J\sigma^{-1})-2)\,c_n\,\sigma T(\sigma^2)\\
&=& 2\,c_n\,\left[ T(\sigma^2)- \sigma T(\sigma^2)\right]
\end{eqnarray*}
which is zero in $\Theta$. Hence $(t-2)\,q_n(t)$ lies in the kernel of $\theta$. 
Since both $W[[t]]/((t-2)\,q_n(t))$ and $\Theta$ are free over $W$ of rank $2^{n-2}$, it follows 
that $W[[t]]/((t-2)\,q_n(t))$ is isomorphic to $\Theta$, which completes the proof of Lemma 
\ref{lem:semidihedraltrick}.
\end{proof}

The following result gives a correction of \cite[Lemma 5.4]{3quat} and generalizes the corrected
result to all tame blocks with at least two isomorphism classes of simple modules.
\begin{lemma}
\label{lem:quaternionargument}
Assume Hypothesis $\ref{hyp:alltheway}$ and use the notation introduced above, and
in particular $(\ref{eq:nu})$, $(\ref{eq:goodchar1})$ and Definition $\ref{def:seemtoneed}$.
Assume $n\ge 3$.
If $D$ is quaternion of order $8$, let $U'$ be a $WG$-module which is free over $W$ and 
whose $F$-character is either $\chi_{5,1}$, or $\chi_6$, or $\chi_7$. If $D$ is not quaternion of
order $8$,
let $U'$ be a $WG$-module which is free over $W$ and whose $F$-character is equal to
$$\sum_{\ell=2}^{n-1} \rho_\ell = \sum_{i=1}^{2^{n-2}-1}\chi_{5,i}.$$
\begin{enumerate}
\item[(i)]
There exists a $WG$-module 
endomorphism $\alpha$ of $U'$ such that the $W$-algebra $W[\alpha]$ generated by $\alpha$
is isomorphic to $R'$.

\item[(ii)] 
Suppose either that $n=3$, or that $n\ge 4$ and
$\mathrm{End}_{kG}(U'/2U')\cong R'/2R'$ and $U'/2U'$ is free as a module
for $\mathrm{End}_{kG}(U'/2U')$ of rank $\mathrm{deg}(\chi_{5,1})$.
Then $\mathrm{End}_{WG}(U')=W[\alpha]\cong R'$ and $U'$ is free as a module 
for $\mathrm{End}_{WG}(U')$.
\end{enumerate}
\end{lemma}

\begin{proof}
Suppose first that $n=3$. Then $q_n(t)=t$ and $R'\cong W$. Hence 
we can choose the $WG$-module endomorphism $\alpha$ of $U'$ to be the zero endomorphism. 
Since $U'$ is free as a $W$-module and $F\otimes_W\mathrm{End}_{WG}(U')\cong
\mathrm{End}_{FG}(F\otimes_W U')\cong F$, 
it follows that $\mathrm{End}_{WG}(U')\cong W \cong R'$.

For the remainder of the proof, assume $n\ge 4$.
We first construct a $WG$-module endomorphism $\alpha$ of
$U'$ as in part (i) of the lemma. 
As before, let $\sigma\in D$ be an element of order $2^{n-1}$ and 
let $\{v_1,\ldots,v_{l(\sigma)}\}$ be a complete system of representatives of 
$C_G(\sigma)$-conjugacy classes of $2$-regular elements in $C_G(\sigma)$. 
For $1\le j\le l(\sigma)$, let $C_j$ be the conjugacy class in $G$ of $\sigma v_j$, and let 
$t(C_j)\in WG$ be the class sum of $C_j$. 

Let $1\le j\le l(\sigma)$. Since $t(C_j)$ lies in the center of $WG$, multiplication by $t(C_j)$
defines a $WG$-module endomorphism  of $U'$. Since $U'$ is free as a $W$-module,
$\mathrm{End}_{WG}(U')$ can naturally be identified with a $W$-subalgebra of 
$$F\otimes_W\mathrm{End}_{WG}(U')\cong
\mathrm{End}_{FG}(F\otimes_W U')\cong \prod_{\ell=2}^{n-1}\mathrm{End}_{FG}(V_\ell)
\cong\prod_{\ell=2}^{n-1}F(\nu_\ell).$$
Therefore, $t(C_j)$ acts on $U'$ as multiplication by a scalar 
$\lambda_j$ in the maximal $W$-order $\prod_{\ell=2}^{n-1} W[\nu_\ell]$ in
$\prod_{\ell=2}^{n-1}F(\nu_\ell)$. 
Moreover, $\lambda_j$ can be determined from the action of $t(C_j)$ on 
$F\otimes_W U'\cong\bigoplus_{\ell=2}^{n-1}V_\ell$, which implies by (\ref{eq:goodone})
that $\lambda_j=\left(\mu_j(\ell)\right)_{\ell=2}^{n-1}$.
By (\ref{eq:duhh2}), there exist elements $w_1,\ldots, w_{l(\sigma)}
\in W$ such that
$$\nu_\ell = w_1\cdot \mu_1(\ell) + \cdots + w_{l(\sigma)} \cdot \mu_{l(\sigma)}(\ell)$$
for $2\le \ell \le n-1$.
Define $\alpha$ to be the $WG$-module endomorphism of $U'$ given by multiplication by
the element $\sum_{j=1}^{l(\sigma)} w_j \,t(C_j)$ in the center of $WG$.
Then $\alpha$ acts on $U'$ as multiplication by the scalar 
$(\nu_\ell)_{\ell=2}^{n-1}\in \prod_{\ell=2}^{n-1} W[\nu_\ell]$, which implies 
$W[\alpha]\cong R'$ (see Remark \ref{rem:ohyeah}).

Let now $\overline{U'}=U'/2U'$ and suppose
that $\mathrm{End}_{kG}(\overline{U'})\cong R'/2R'\cong k[t]/(t^{2^{n-2}-1})$ 
and that $\overline{U'}$ is free as a module
for $\mathrm{End}_{kG}(U'/2U')$ of rank $x=\mathrm{deg}(\chi_{5,1})$.
Note that $x=\mathrm{deg}(\chi_{5,i})$ for all $1\le i\le 2^{n-2}-1$ and that
$\mathrm{dim}_{F(\nu_\ell)}V_\ell = x$ for all $2\le \ell \le n-1$, which implies that
$F\otimes_W U'$ is a free module of rank $x$ for $\mathrm{End}_{FG}(F\otimes_W U')$.
We have a short exact sequence of $W$-modules
$$\xymatrix
{0\ar[r] &\mathrm{End}_W(U') \ar[r]^{\cdot 2}& \mathrm{End}_W(U')
\ar[r]^{\mathrm{mod}\;2} & \mathrm{End}_k(\overline{U'})\ar[r]&0.}$$
Considering the $G$-action on these modules, we obtain an exact sequence of $W$-modules
\begin{equation}
\label{eq:cohom}
\xymatrix
{0\ar[r] &\mathrm{End}_{WG}(U') \ar[r]^{\cdot 2}& \mathrm{End}_{WG}(U')
\ar[r]^{\mathrm{mod}\;2} & \mathrm{End}_{kG}(\overline{U'})\ar[r]&
\HH^1(G,\mathrm{End}_W(U')).}
\end{equation}
Because $\mathrm{End}_{WG}(U')$ is a $W$-submodule of the free $W$-module
$\mathrm{End}_W(U')$, it follows that $\mathrm{End}_{WG}(U')$ is a free
$W$-module of rank 
$$\mathrm{dim}_F (F\otimes_W\mathrm{End}_{WG}(U')) =
\mathrm{dim}_F\; \mathrm{End}_{FG}(F\otimes_W U')=2^{n-2}-1.$$
Since by assumption $\mathrm{dim}_k\,
\mathrm{End}_{kG}(\overline{U'})=\mathrm{dim}_k (R'/2R')=2^{n-2}-1$, it follows that
(\ref{eq:cohom}) induces a short exact sequence of $W$-modules
\begin{equation}
\label{eq:cohom1}
\xymatrix
{0\ar[r] &\mathrm{End}_{WG}(U') \ar[r]^{\cdot 2}& \mathrm{End}_{WG}(U')
\ar[r]^{\mathrm{mod}\;2} & \mathrm{End}_{kG}(\overline{U'})\ar[r]&0.}
\end{equation}
Let $\overline{\beta}$ be a generator of $\mathrm{End}_{kG}(\overline{U'})$ as a 
$k$-algebra. By (\ref{eq:cohom1}), there exists $\beta\in \mathrm{End}_{WG}(U')$
whose induced $kG$-module endomorphism of $\overline{U'}$ is equal to
$\overline{\beta}$. Using Nakayama's lemma, we see that $\mathrm{End}_{WG}(U')=
W[\beta]$. Since the maximal ideal $\mathfrak{m}_{W[\beta]}$ of $W[\beta]$ is generated by 2 and 
$\beta$ and the maximal ideal $\mathfrak{m}_{k[\overline{\beta}]}$ of $k[\overline{\beta}]$ is 
generated by $\overline{\beta}$, it follows that
$$U'/\mathfrak{m}_{W[\beta]}U' \cong \overline{U'}/\overline{\beta}(\overline{U'})
\cong  \overline{U'}/\mathfrak{m}_{k[\overline{\beta}]}\overline{U'}$$
where the latter has $k$-dimension $x=\mathrm{deg}(\chi_{5,1})$ by assumption.
By Nakayama's lemma, we can lift a $k$-basis 
$\{\overline{s}_1,\ldots, \overline{s}_x\}$ of $U'/\mathfrak{m}_{W[\beta]}U'$  
to a set  $\{s_1,\ldots, s_x\}$ of generators of $U'$ over $W[\beta]=\mathrm{End}_{WG}(U')$. 
Because $F\otimes_W U'$ is a free module of rank $x$ for 
$\mathrm{End}_{FG}(F\otimes_W U')\cong F\otimes_W \mathrm{End}_{WG}(U')$,
it follows that $s_1,\ldots, s_x$ are linearly independent over $\mathrm{End}_{WG}(U')$. 
Hence $U'$ is free as a module for $\mathrm{End}_{WG}(U')$ of rank
$x=\mathrm{deg}(\chi_{5,1})$.

It remains to show that $\mathrm{End}_{WG}(U')=W[\alpha]$, i.e. we need to show 
$W[\alpha]=W[\beta]$. We identify both $W[\alpha]$ and $W[\beta]$ with $W$-subalgebras of  the 
maximal $W$-order $\prod_{\ell=2}^{n-1} W[\nu_\ell]$ in
$\prod_{\ell=2}^{n-1}F(\nu_\ell)\cong F\otimes_W \mathrm{End}_{WG}(U')$. 
Since $W[\alpha]\subseteq W[\beta]$, there exists a polynomial $q(X)\in W[X]$ such that 
$\alpha=q(\beta)$. Because $\alpha$ is not a unit, the constant coefficient of $q(X)$ must
be divisible by 2. Write
$$q(X)=2\tilde{a}_0+a_1X+a_2X^2+\cdots+a_dX^d$$
for certain $\tilde{a}_0,a_1,\ldots,a_d\in W$ and $d\ge 1$. 
Suppose $a_1=2\tilde{a}_1$ for some $\tilde{a}_1\in W$. Consider the natural projection
$$\pi_{n-1}:\prod_{\ell=2}^{n-1} W[\nu_\ell] \to W[\nu_{n-1}],$$
and let $\mathfrak{m}$ be the maximal ideal of $W[\nu_{n-1}]$. Since $n\ge 4$, either
$\sqrt{2}$ or $\sqrt{-2}$ lies in $\mathfrak{m}$. Hence $2\in\mathfrak{m}^2$. 
Since $2\pi_{n-1}(\beta),\pi_{n-1}(\beta)^2,\ldots,\pi_{n-1}(\beta)^d$ all lie in 
$\mathfrak{m}^2$, it follows that $\pi_{n-1}(\alpha)\in\mathfrak{m}^2$. However, 
we have seen that $\pi_{n-1}(\alpha)=\nu_{n-1}\not\in\mathfrak{m}^2$. Hence
$a_1$ is not divisible by 2. But then the $kG$-module endomorphism $\overline{\alpha}$ of 
$\overline{U'}$ which is induced by $\alpha$ has the form
$$\overline{\alpha}=\overline{a}_1\overline{\beta} + \overline{a}_2\overline{\beta}^2 + \cdots +
\overline{a}_d \overline{\beta}^d$$
where $\overline{a}_1\in k^*$ and $\overline{a}_2,\ldots,\overline{a}_d\in k$. This means that 
$k[\overline{\alpha}]=k[\overline{\beta}]=\mathrm{End}_{kG}(\overline{U'})$, which implies by Nakayama's
lemma that $\mathrm{End}_{WG}(U')=W[\alpha]$.
\end{proof}


\section{Universal deformation rings}
\label{s:udr}

Assume Hypothesis \ref{hyp:alltheway} and the notation introduced in Section \ref{s:tame}
and Section \ref{s:ordinarytame}. In particular,
$p=2$, $n\ge 2$, and $D$  is dihedral, semidihedral or generalized quaternion. 
In this section, we want to determine $R(G,V)$ for any 
maximally ordinary $kG$-module
$V$ in the sense of Definition \ref{def:maximallyordinary}.

We first determine all indecomposable $kG$-modules $V$ belonging to $B$ whose stable 
endomorphism rings are isomorphic to $k$ and whose Brauer characters are restrictions
of ordinary irreducible characters of height 1. By Remark \ref{rem:cases}, the ordinary irreducible 
characters of height 1 belonging to $B$ are (using the notation from Section \ref{s:ordinarytame}):
\begin{itemize}
\item none if $n=2$,
\item $\chi_{5,1},\chi_6,\chi_7$ if $D$ is quaternion of order 8, and
\item $\chi_{5,i}$ for all $1\le i\le 2^{n-2}-1$ if $n\ge 3$ and $D$ is not quaternion of order 8.
\end{itemize}

\begin{lemma}
\label{lem:biglist}
Assume Hypothesis $\ref{hyp:alltheway}$, and assume $n\ge 3$. 
Let $\Lambda=kQ/I$ be a basic algebra such that
$B$ is Morita equivalent to $\Lambda$. For each vertex $j$ in $Q$, let $T_j$ denote the
simple $B$-module corresponding to the simple $\Lambda$-module belonging to $j$.
Let $V$ be an indecomposable $kG$-module 
belonging to $B$ such that $\underline{\mathrm{End}}_{kG}(V)\cong k$ and such that
the Brauer character of $V$ is equal to the restriction to the $2$-regular conjugacy classes 
of an ordinary irreducible character of $G$ of height $1$.
\begin{enumerate}
\item[(a)] Suppose $D$ is quaternion of order $8$.
If $Q=3\mathcal{A}$ or $Q=3\mathcal{B}$ then $V$ is either isomorphic to $T_1$ or $T_2$, or
$V$ is a uniserial module of length $4$ whose radical quotient or socle is isomorphic to $T_0$.
If $Q=3\mathcal{K}$ then $V$ is an arbitrary uniserial module 
of length $2$.
\item[(b)] Suppose $D$ is not quaternion of order $8$.
\begin{enumerate}
\item[(i)] If $Q=2\mathcal{A}$ then $V$ is a uniserial module with descending composition factors
$T_0,T_0,T_1$, or $T_1,T_0,T_0$.
\item[(ii)] If $Q=2\mathcal{B}$ and $B$ is Morita equivalent to neither  
$\SD(2\mathcal{B})_4(c)$ nor $\Q(2\mathcal{B})_2(p,a,c)$ then $V$ is isomorphic to $T_1$. If
$B$ is Morita equivalent to $\SD(2\mathcal{B})_4(c)$ or $\Q(2\mathcal{B})_2(p,a,c)$ then $V$ is
a uniserial module 
with descending composition factors $T_0,T_1$, or $T_1,T_0$.
\item[(iii)] If $Q=3\mathcal{A}$ then $V$ is a uniserial module 
with descending composition factors
$T_0,T_1,T_0,T_2$, or $T_2,T_0,T_1,T_0$, or $T_0,T_2,T_0,T_1$, or $T_1,T_0,T_2,T_0$.
\item[(iv)] If $Q\in\{3\mathcal{B},3\mathcal{D}\}$ then $V$ is isomorphic to $T_1$.
\item[(v)] If $B$ is Morita equivalent to $\SD(3\mathcal{C})_{2,1}$ then
$V$ is isomorphic to $T_0$. If $B$ Morita equivalent to $\SD(3\mathcal{C})_{2,2}$ then
$V$ is indecomposable with descending radical factors $T_0,T_1\oplus T_2$, or 
$T_1\oplus T_2,T_0$,
or $V$ is uniserial with descending composition factors $T_1,T_0,T_2$, or $T_2,T_0,T_1$.
\item[(vi)] If $Q\in\{3\mathcal{H},3\mathcal{K}\}$ and $B$ is not Morita equivalent to
$\SD(3\mathcal{H})_2$ then $V$ is a uniserial module with 
descending composition factors $T_1,T_2$, or $T_2,T_1$.
If $B$ is Morita equivalent to $\SD(3\mathcal{H})_2$ then $V$ is a uniserial module with 
descending composition factors $T_0,T_1$, or $T_1,T_0$.
\end{enumerate}
\end{enumerate}
Conversely, if $V$ is as in $(a)$ or $(b)$, then $\mathrm{End}_{kG}(V)\cong k$ and
the Brauer character of $V$ is equal to the restriction to the $2$-regular conjugacy classes 
of an ordinary irreducible character of height $1$.
Moreover, $\mathrm{Ext}^1_{kG}(V,V)=0$ if  $D$ is quaternion of order $8$,
and $\mathrm{Ext}^1_{kG}(V,V)\cong k$ in all other cases.
\end{lemma}

\begin{proof}
Lemma \ref{lem:biglist} is proved using the description of the basic algebra $\Lambda=kQ/I$ of the
block $B$ together with its decomposition matrix as provided in Section \ref{s:tame}.
To give an idea of the arguments, we discuss the case when $\Lambda=\SD(3\mathcal{A})_1$
in part (b)(iii).
It follows from the decomposition matrix in Figure \ref{fig:decompSD(3A)} that $V$ is an
indecomposable $B$-module with composition factors $T_0,T_0,T_1,T_2$.
Let $t_V$ (resp. $s_V$) be the composition series length of the radical quotient (resp. socle) of $V$.
If $t_V=3$ then $V$ has radical series length 2 and $s_V=1$. But then $V$ is a submodule of a projective 
indecomposable $B$-module which is impossible for $t_V=3$. If $t_V=2$ then $V$ has radical series length 
at most 3 and $s_V\le 2$. Using that $\mathrm{Ext}^1_B(T_i,T_j)$ is one-dimensional over $k$ 
if $(i,j)\in\{(0,1),(0,2),(1,0),(2,0)\}$ and zero otherwise, it follows that there are no indecomposable
$B$-modules $V$ with $t_V=2$ and radical series length 2. Analyzing $\mathrm{Ext}^1_B(T_{i_1}\oplus
T_{i_2}, V')$ for appropriate indecomposable $B$-modules $V'$ of length 2, we see that the only indecomposable 
$B$-modules $V$ with $t_V=2$ and radical series length 3 are submodules of the projective cover
$P_{T_0}$ 
of the form
$$\begin{array}{c@{}cc}T_0\\&T_1&T_2\\&\multicolumn{2}{c}{T_0}\end{array}
\qquad\mbox{ or }\qquad \begin{array}{cc@{}c}&&T_0\\T_1&T_2\\\multicolumn{2}{c}{T_0}\end{array}\;.$$
However, the stable endomorphism ring of each of these modules has $k$-dimension 2.
If $t_V=1$ then $V$ has radical series length at most 4 and $s_V\le 3$. Since $V$ is a quotient module
of a projective indecomposable $B$-module in this case, there are no indecomposable $B$-modules
$V$ with $t_V=1$ and radical series length 2. If $t_V=1$ and the radical series length of $V$ is 3, we
use similar $\mathrm{Ext}^1$ arguments as above to see that the only indecomposable $B$-modules
$V$ with these properties are quotient modules of the projective cover $P_{T_0}$ 
with descending radical factors $T_0, T_1\oplus T_2,T_0$.
But all such modules have an endomorphism which factors through $T_0$ and which does not factor
through a projective module, meaning that
the stable endomorphism ring of each of these modules has 
$k$-dimension 2. Finally, if $t_V=1$ and the radical series length of $V$ is 4, then $V$ is one of
the uniserial modules described in part (b)(iii) of Lemma \ref{lem:biglist}. This description shows directly that 
$\mathrm{End}_B(V)\cong k$. Moreover, using the projective indecomposable 
$B$-modules, we see that $\mathrm{Ext}^1_B(V,V)\cong k$.
\end{proof}

\begin{cor}
\label{cor:biglist}
Assume Hypothesis $\ref{hyp:alltheway}$. 
Let $V$ be an indecomposable $kG$-module 
belonging to $B$ such that $\underline{\mathrm{End}}_{kG}(V)\cong k$.
The following statements are equivalent:
\begin{enumerate}
\item[(i)] $V$ is maximally ordinary in the sense of Definition $\ref{def:maximallyordinary}$;
\item[(ii)] $n\ge 4$ and the Brauer character of $V$ is equal to the restriction to the $2$-regular 
conjugacy classes of an ordinary irreducible character of $G$ of height $1$;
\item[(iii)] $n\ge 4$ and $V$ is as in part $(b)$ of Lemma $\ref{lem:biglist}$.
\end{enumerate}
\end{cor}

\begin{proof}
If $n=2$ (resp. $n=3$), it follows from \cite[Sect. VII]{brIV} (resp. \cite{brauer2,olsson})
that all generalized decomposition numbers corresponding to maximal $2$-power
order elements in $D$ lie in $\{0,\pm 1\}$. 
Hence there are no maximally ordinary $kG$-modules belonging to $B$ if $n\le 3$.

Now suppose $n\ge 4$. Let $\sigma\in D$ be of maximal $2$-power order, and
let $\chi$ be an ordinary irreducible character belonging to $B$.
It follows from \cite{brauer2,olsson} that if $\chi$ has height 0, then the non-zero
generalized decomposition numbers corresponding to $\sigma$ and $\chi$
are $\pm 1$. Also, if $\chi$ has height $n-2$, then all
generalized decomposition numbers corresponding to $\sigma$ and $\chi$ are zero.
On the other hand, we see from (\ref{eq:great0}), since $n\ge 4$, that if $\chi$ has height 1, then
$d_{\chi,\varphi^{(\sigma)}}^{\,(\sigma)}\not\in\{0,\pm 1\}$ 
if and only if 
$\chi=\chi_{5,i}$ for $i\in\{1,2,\ldots,2^{n-2}-1\}-\{2^{n-3}\}$. Note that the restriction of
$\chi_{5,i}$ to the $2$-regular conjugacy classes is the same for all $1\le i\le 2^{n-2}-1$.
Therefore it follows that if $n\ge 4$ then $V$ is maximally ordinary if and only if its 
Brauer character is equal to the restriction to the $2$-regular conjugacy 
classes of an ordinary irreducible character of height 1. 
Hence Corollary \ref{cor:biglist}  follows from Lemma \ref{lem:biglist}.
\end{proof}

The lifts of some of the $kG$-modules $V$ from Lemma \ref{lem:biglist} are connected to 3-tubes in the 
stable Auslander-Reiten quiver of $B$ as follows.

\begin{dfn}
\label{def:mixedcase}
Assume Hypothesis \ref{hyp:alltheway}, assume $n\ge 3$,
and let $V$ be as in Lemma \ref{lem:biglist}.
We say $V$ \emph{corresponds to a $3$-tube} if  
there exists an indecomposable quotient module $\overline{U}$ of the projective 
$kG$-module cover $P_V$ of $V$ such that 
\begin{enumerate}
\item[(a)] $\overline{U}$ defines a lift of $V$ over $k[t]/(t^{2^{n-2}})$, and
\item[(b)] $\overline{U}$ belongs to a $3$-tube of the stable Auslander-Reiten quiver of $B$.
\end{enumerate}
\end{dfn}

The following lemma determines which modules $V$ from Lemma $\ref{lem:biglist}$ correspond to 3-tubes.

\begin{lemma}
\label{lem:mixedcase}
Assume Hypothesis $\ref{hyp:alltheway}$, assume $n\ge 3$, 
and let $V$ be as in Lemma $\ref{lem:biglist}$. 
If $D$ is dihedral then $V$ corresponds to a $3$-tube, and if $D$
is generalized quaternion then $V$ does not correspond to a $3$-tube.
If $D$ is semidihedral then $V$ corresponds to a $3$-tube if and only if
either 
\begin{itemize}
\item $B$ is Morita equivalent to one of the algebras in 
$$\{\SD(2\mathcal{A})_2(c),\SD(2\mathcal{B})_1(c),\SD(3\mathcal{B})_1,
\SD(3\mathcal{C})_{2,1}\}$$
and $V$ is arbitrary; or
\item $B$ is Morita equivalent to $\SD(3\mathcal{A})_1$ 
and $V$ is  such that its radical quotient or its socle
is isomorphic to $T_1$; or
\item $B$ is Morita equivalent to $\SD(3\mathcal{C})_{2,2}$ 
and $V$ is such that its radical quotient or its socle
is isomorphic to $T_0$; or
\item $B$ is Morita equivalent to $\SD(3\mathcal{H})_1$ and $V$ is such that 
its radical quotient is isomorphic to $T_1$;
\item $B$ is Morita equivalent to $\SD(3\mathcal{H})_2$ and $V$ is such that 
its radical quotient is isomorphic to $T_0$.
\end{itemize}
In all cases, if $V$ corresponds to a $3$-tube, then $\overline{U}$ from Definition $\ref{def:mixedcase}$ 
belongs to the boundary of its $3$-tube.
\end{lemma}

\begin{proof}
If $D$ is generalized quaternion, Lemma \ref{lem:mixedcase} follows from 
\cite[V.4.3]{erd}.

Suppose next that $D$ is dihedral, and let $\Lambda=kQ/I$ be a basic algebra such that $B$ is Morita 
equivalent to $\Lambda$. By \cite[VI.10.1]{erd}, it follows that $\Lambda/\mathrm{soc}(\Lambda)$ is
special biserial. Therefore, we can use the techniques described in \cite[Sect. 3]{buri} to determine
the 3-tubes of the stable Auslander-Reiten quiver of $B$. Note that the modules at the boundaries
of the 3-tubes are either maximal uniserial or simple. Using the list of the possible $V$ in Lemma
\ref{lem:biglist}, we see by direct inspection that $V$ always corresponds to a 3-tube and
that the module $\overline{U}$ from Definition $\ref{def:mixedcase}$ 
always belongs to the boundary of its $3$-tube.

Finally, suppose that $D$ is semidihedral. By \cite[V.4.2]{erd}, the stable Auslander-Reiten quiver
of $B$ has at most one 3-tube. 
To give an idea of the arguments, we discuss the case when $B$ is Morita equivalent to 
$\Lambda=\SD(3\mathcal{A})_1$.
By Lemma \ref{lem:biglist}(b)(iii), $V$ is a uniserial module with descending composition factors
$T_0,T_1,T_0,T_2$, or $T_2,T_0,T_1,T_0$, or $T_0,T_2,T_0,T_1$, or $T_1,T_0,T_2,T_0$.
Since $V$ is uniserial with $\mathrm{End}_{kG}(V)\cong k\cong\mathrm{Ext}^1_{kG}(V,V)$, 
we can use \cite[Lemma 2.5]{3quat} to show that $R(G,V)/2R(G,V)$ is isomorphic to $k[t]/(t^{2^{n-2}})$
(resp. $k[t]/(t^{2^{n-2}-1})$) if the radical quotient or the socle of $V$ is isomorphic to $T_1$
(resp. $T_2$). Hence by Definition \ref{def:mixedcase}, $V$ can only correspond to a 3-tube
if the radical quotient or the socle of $V$ is isomorphic to $T_1$. Since the projective cover
$P_{T_1}$ is uniserial, we obtain from \cite[Lemma 2.5]{3quat} that $\overline{U}=\Omega^{-1}(T_1)$
(resp. $\overline{U}=\Omega(T_1)$) defines a lift of $V$ over $k[t]/(t^{2^{n-2}})$
if the radical quotient (resp. the socle) of $V$ is isomorphic to $T_1$. Since
$\Omega^3(\overline{U})\cong \overline{U}$, it follows that $\overline{U}$ belongs to a
3-tube of the stable Auslander-Reiten quiver of $B$. Moreover, using the criterion given in 
\cite[Sect. 1]{buri} for almost split sequences to have an indecomposable middle term, we see
that $\overline{U}$ belongs to the boundary of this 3-tube.
\end{proof}

The following result shows that if  $V$ corresponds to a 3-tube, then
the module $\overline{U}$ from Definition \ref{def:mixedcase} has a universal deformation
ring $R(G,\overline{U})\cong k$.

\begin{prop}
\label{prop:3tubes}
Assume Hypothesis $\ref{hyp:alltheway}$, assume $n\ge 3$,
and let $V$ be as in Lemma $\ref{lem:biglist}$. 
Moreover suppose that $V$ corresponds to a $3$-tube.
Let $\overline{U}$ be the $kG$-module from Definition $\ref{def:mixedcase}$ which
belongs to a $3$-tube of the stable Auslander-Reiten quiver of $B$. Then 
$\underline{\mathrm{End}}_{kG}(\overline{U})\cong k$ and $R(G,\overline{U})\cong k$.
\end{prop}

\begin{proof}
If $D$ is dihedral, this follows from \cite[Sect. 5.2]{3sim} and \cite[Prop. 6.3]{2sim}.
For the remainder of the proof, assume that $D$ is semidihedral.

Let $\mathfrak{T}$ be the 3-tube of the stable Auslander-Reiten quiver of $B$ to which 
$\overline{U}$ belongs. By Lemma \ref{lem:mixedcase}, 
$\overline{U}$ belongs to the boundary of $\mathfrak{T}$. Going through the cases
described in Lemma \ref{lem:mixedcase}, it is straightforward to show that 
$\underline{\mathrm{End}}_{kG}(\overline{U})\cong k$ and $\mathrm{Ext}^1_{kG}(
\overline{U},\overline{U})=0$.

Let $K$ be a vertex of $\overline{U}$. Because of the work in  \cite[Chapter V]{erd},
and in particular \cite[V.4.2.1 and proof of V.4.2]{erd}, we 
have the following facts:
\begin{enumerate}
\item[(i)] The group $K$ is a Klein four group and the quotient group $N_G(K )/C_G(K )$ is 
isomorphic to a symmetric group $S_3$ . 
\item[(ii)] There is a block $b$ of $kN_G(K )$ with $b^G = B$ such that the Green correspondent 
$f \overline{U}$ of $\overline{U}$
belongs to the boundary of a 3-tube in the stable Auslander-Reiten quiver of $b$. Moreover, 
$b$ is Morita equivalent to $kS_4$ modulo the socle. 
\end{enumerate}
Using these facts in lieu of \cite[Facts 5.2.1]{3sim}, we can use similar arguments as in 
the proofs of \cite[Prop. 5.2.4 and Cor. 5.2.5]{3sim} 
to show that $R(G, \overline{U}) \cong k$. 
\end{proof}

By Corollary \ref{cor:biglist}, Theorem \ref{thm:supermainNEW} is a consequence of the following
result.

\begin{thm}
\label{thm:main}
Assume Hypothesis $\ref{hyp:alltheway}$, and assume $n\ge 3$.
Let $V$ be an indecomposable $kG$-module 
belonging to $B$ such that $\underline{\mathrm{End}}_{kG}(V)\cong k$ and such that
the Brauer character of $V$ is equal to the restriction to the $2$-regular conjugacy classes 
of an ordinary irreducible character of $G$ of height $1$.
Then
$$R(G,V)\cong \left\{\begin{array}{cl} 
W[[t]]/(t\,q_n(t),2\,q_n(t))&\mbox{ if $V$ corresponds to a $3$-tube,}\\
W[[t]]/(q_n(t))&\mbox{ if $V$ does not correspond to a $3$-tube,}\end{array}\right.$$
where $q_n(t)$ is as in Definition $\ref{def:seemtoneed}$.
In all cases, the ring $R(G,V)$ is a subquotient ring of $WD$, and it is a complete intersection 
ring if and only if $V$ does not correspond to a $3$-tube.
\end{thm}

\begin{proof}
By Lemma \ref{lem:biglist}, $V$ is one of the modules in Lemma \ref{lem:biglist}, listed in
part (a) or part (b).

Suppose first that $D$ is quaternion of order 8. By Lemma \ref{lem:biglist}(a), 
$\mathrm{Ext}^1_{kG}(V,V)=0$, which implies that $R(G,V)$ is isomorphic to a quotient
algebra of $W$. Using the decomposition matrices in Figures \ref{fig:decompQ(3A)}, 
\ref{fig:decompQ(3B)} and \ref{fig:decompQ(3K)} together with \cite[Prop. (23.7)]{CR},
we see that in all cases $V$ can be lifted over $W$, which implies that 
$R(G,V)\cong W\cong W[[t]]/(q_3(t))$. 

For the remainder of the proof, assume that $D$ is not quaternion of order 8.
In particular, the Brauer character of $V$ is the restriction of $\chi_{5,1}$ to the
2-regular conjugacy classes of $G$.
By Lemma \ref{lem:biglist}(b), $\mathrm{Ext}^1_{kG}(V,V)\cong k$,
which implies that $R(G,V)$ is isomorphic to a quotient algebra of $W[[t]]$ but not
to a quotient algebra of $W$. Let $P_V$ be the projective $kG$-module cover of $V$.

\medskip

\textit{Claim $0$.} There exists an indecomposable quotient module $\overline{U'}$ of $P_V$ such that 
$\overline{U'}$ defines a lift of $V$ over $k[t]/(t^{2^{n-2}-1})$ and
$\overline{U'}/t^2\,\overline{U'}$ is an indecomposable $kG$-module.
Moreover, if $V$ corresponds to a 3-tube then the indecomposable $kG$-module
$\overline{U}$ of Definition \ref{def:mixedcase}, which is also an indecomposable quotient
module of $P_V$, defines a lift of $V$ over $k[t]/(t^{2^{n-2}})$ and
satisfies $\overline{U}/t^{2^{n-2}-1}\,\overline{U}\cong \overline{U'}$.

\medskip

\textit{Proof of Claim $0$.}
Suppose first that $B$ is Morita equivalent neither to $\SD(2\mathcal{B})_4(c)$, nor to 
$\Q(2\mathcal{B})_2(p,a,c)$, nor to $\SD(3\mathcal{C})_{2,2}$. Then $V$ is uniserial of length
$\ell\le 4$. Using the quiver and relations of the basic algebra of $B$, as provided in Section \ref{s:tame},
it follows that in all cases $P_V$ has an indecomposable quotient module
$\overline{U'}$ (resp. $\overline{U}$, provided it exists) which is uniserial of
length $\ell\,(2^{n-2}-1)$ (resp. $\ell\,2^{n-2}$). Moreover, $\overline{U'}$ (resp. $\overline{U}$)
can be pictured as having $2^{n-2}-1$ (resp. $2^{n-2}$) copies of $V$ stacked
on top of each other. This means that the action of $t$ on $\overline{U'}$ (resp. $\overline{U}$)
is given by an automorphism of $\overline{U'}$ (resp. $\overline{U}$) which is
unique up to multiplication by a non-zero scalar and which factors through $\mathrm{rad}^\ell(\overline{U'})$ 
(resp. $\mathrm{rad}^\ell(\overline{U})$). This implies Claim 0 in this case.

If $B$ is Morita equivalent to $\SD(2\mathcal{B})_4(c)$ or $\Q(2\mathcal{B})_2(p,a,c)$,
it follows from Lemma \ref{lem:biglist}(b)(ii) that $V$ is uniserial of length 2 with descending 
composition factors $T_u,T_v$, where $\{u,v\}=\{0,1\}$. Using the relations in Section
\ref{ss:quiver2B}, we see
that the projective cover $P_V=P_{T_u}$ 
has a unique submodule $K_u$ which is uniserial with descending composition factors $T_u,T_u$.
It follows that $\overline{U'}=P_{T_u}/K_u$, which can be
visualized as in (\ref{eq:lacy2B}), where $T_u$ (resp. $T_v$) occurs $2^{n-2}-1$ times.
This implies Claim 0 in this case, since $V$ does not correspond to a 3-tube, so
$\overline{U}$ does not exist.
\begin{equation}
\label{eq:lacy2B}
\overline{U'}=\vcenter{ \xymatrix @R=.1pc @C=.2pc{&T_u&\\
T_u\ar@{-}[rrd]&&T_v\ar@{-}[lld]\\T_u\ar@{-}[rrd]&&T_v\ar@{-}[lld]\\T_u&&T_v\\ \vdots&&\vdots\\
T_u\ar@{-}[rrd]&&T_v\ar@{-}[lld]\\T_u&&T_v\\&T_v&}}
\end{equation}

Finally, suppose $B$ is Morita equivalent to $\SD(3\mathcal{C})_{2,2}$. By Lemma \ref{lem:biglist}(b)(v),
$V$ is indecomposable with descending radical factors $T_0,T_1\oplus T_2$, or $T_1\oplus T_2,T_0$,
or $V$ is uniserial with descending composition factors $T_1,T_0,T_2$, or $T_2,T_0,T_1$.
Using the relations in Section \ref{ss:quiver3C}, we see that
there exists a unique uniserial $B$-module $T_{00}$, up to isomorphism, with descending composition
factors $T_0,T_0$. Also, for $u\in\{1,2\}$, there exists a unique uniserial $B$-module
$T_{u0u}$, up to isomorphism, with descending composition factors $T_u,T_0,T_u$. 
Moreover, the projective cover $P_{T_0}$ has a unique submodule isomorphic to $T_{00}$, and
for $u\in\{1,2\}$, the projective cover $P_{T_u}$ has a unique submodule isomorphic to $T_{u0u}$.
If the radical quotient (resp. socle) of $V$  is $T_0$, it follows that $\overline{U}$ is 
isomorphic to $\Omega^{-1}(T_{00})$ (resp. $\Omega(T_{00})$), and $\overline{U'}$ is 
isomorphic to $\mathrm{rad}^2(\overline{U})$. 
In particular, the action of $t$ on $\overline{U'}$ (resp. $\overline{U}$)
is given by an automorphism of $\overline{U'}$ (resp. $\overline{U}$) which is
unique up to multiplication by a non-zero scalar and which factors through $\mathrm{rad}^2(\overline{U'})$ 
(resp. $\mathrm{rad}^2(\overline{U})$).
If the radical quotient of $V$ is
$T_u$ for $u\in\{1,2\}$, it follows that $\overline{U}$ does not exist and $\overline{U'}$ is isomorphic to
$\Omega^{-1}(T_{u0u})$. This completes the proof of Claim 0.

\medskip

\textit{Claim $1$.} The universal mod 2 deformation ring $R(G,V)/2R(G,V)$ is 
isomorphic to $k[t]/(t^{2^{n-2}})$
(resp. $k[t]/(t^{2^{n-2}-1})$) and the universal mod 2 deformation of $V$ is isomorphic to
$\overline{U}$ (resp. $\overline{U'}$) if $V$ corresponds to a 3-tube (resp. does not
correspond to a 3-tube).

\medskip

\textit{Proof of Claim $1$.}
Suppose first that $B$ is Morita equivalent neither to $\SD(2\mathcal{B})_4(c)$, nor to 
$\Q(2\mathcal{B})_2(p,a,c)$, nor to $\SD(3\mathcal{C})_{2,2}$. As seen in the proof of Claim 0, 
$V$ is uniserial in this case. Moreover, it is straightforward to check that the projective cover of $V$ 
satisfies the hypotheses of \cite[Lemma 2.5]{3quat}. Hence we can use \cite[Lemma 2.5]{3quat} to 
prove Claim 1 in this case.

If $B$ is Morita equivalent to $\SD(2\mathcal{B})_4(c)$ or $\Q(2\mathcal{B})_2(p,a,c)$,
it follows from Lemma \ref{lem:biglist}(b)(ii) that $V$ is uniserial with descending 
composition factors $T_u,T_v$ where $\{u,v\}=\{0,1\}$. 
As seen in the proof of Claim 0, the projective cover $P_V=P_{T_u}$ has a unique submodule 
$K_u$ which is uniserial with descending composition factors $T_u,T_u$, and 
$\overline{U'}=P_{T_u}/K_u$ can be visualized as in (\ref{eq:lacy2B}).
Using the relations in Section \ref{ss:quiver2B}, we see that there is a unique $B$-submodule $V'$ of 
$\overline{U'}$ which is isomorphic to $V$. Moreover, if $T=\overline{U'}/V'$ then there is a
unique $B$-submodule $T'$ of $\overline{U'}$ such that there are $kG$-module isomorphisms
$\varphi:\overline{U'}/T'\to V$ and $\psi:\overline{U'}/V'=T\to T'$. Since
$\mathrm{Ext}^1_{kG}(\overline{U'},V)=0$ and since every surjective $kG$-module
homomorphism $\overline{U'}\to V$ must have kernel equal to $T$, we can argue as
in the proof of \cite[Lemma 2.5]{3quat} to show that $R(G,V)/2R(G,V)$ is isomorphic to 
$k[t]/(t^{2^{n-2}-1})$ and that the universal mod 2 deformation of $V$ is isomorphic to 
$\overline{U'}$. This implies Claim 1 in this case.

If $B$ is Morita equivalent to $\SD(3\mathcal{C})_{2,2}$, we use again
similar arguments as in the proof of \cite[Lemma 2.5]{3quat} to prove Claim 1.
The main point is that, as seen in the proof of Claim 0, 
$\overline{U'}$ and $\overline{U}$ are suitable submodules (resp. 
suitable quotient modules) of the projective cover $P_{T_0}$ if the radical quotient (resp. socle)
of $V$ is $T_0$, and $\overline{U'}$ is a suitable quotient module of $P_{T_u}$ if the
radical quotient of $V$ is $T_u$ for $u\in\{1,2\}$. This completes the proof of Claim 1.

\medskip

\textit{Claim $2$.} In all cases, $\mathrm{End}_{kG}(\overline{U'})\cong k[t]/(t^{2^{n-2}-1})$.

\medskip

\textit{Proof of Claim $2$.} Suppose first that $B$ is Morita equivalent neither to 
$\SD(2\mathcal{B})_4(c)$, nor to $\Q(2\mathcal{B})_2(p,a,c)$, nor to $\SD(3\mathcal{C})_{2,2}$.
As seen in the proof of Claim 0, 
$\overline{U'}$ is a uniserial module which can be pictured as having $2^{n-2}-1$ copies
of $V$ stacked on top of each other.
Since $\overline{U'}$ is isomorphic to a tree module in the sense of \cite{krause}, we 
can use the main result of \cite{krause} to prove Claim 2.

If $B$ is Morita equivalent to $\Lambda\in\{\SD(2\mathcal{B})_4(c),\Q(2\mathcal{B})_2(p,a,c),
\SD(3\mathcal{C})_{2,2}\}$, we use the relations in Section \ref{ss:quiver2B} and Section \ref{ss:quiver3C}
to analyze the possible $kG$-module endomorphisms of $\overline{U'}$, where we use the
description of $\overline{U'}$ as given in the proof of Claim 0.
Using explicit $k$-bases for the $\Lambda$-module $\overline{U'_\Lambda}$ corresponding to 
$\overline{U'}$ under the Morita equivalence, 
a straightforward linear algebra calculation shows that 
$\mathrm{End}_{\Lambda}(\overline{U'_\Lambda})\cong k[t]/(t^{2^{n-2}-1})$. This proves
Claim 2.

\medskip

\textit{Claim $3$.} In all cases, $\overline{U'}$ has a lift
$U'$ over $W$ such that the $F$-character of $U'$ is equal to
$$\sum_{\ell=2}^{n-1}\rho_\ell=\sum_{i=1}^{2^{n-2}-1}\chi_{5,i}$$
where $\rho_\ell$ is as in (\ref{eq:goodchar1}).

\medskip

\textit{Proof of Claim $3$.}
In all cases, we use the description of $\overline{U'}$ as given in the proof of Claim 0.
Suppose first that either $B$ is not Morita equivalent to any of the algebras in 
$$\{\SD(2\mathcal{A})_1(c),\Q(2\mathcal{A})(c),\SD(3\mathcal{A})_1,\Q(3\mathcal{A})_2\}$$
and $V$ is arbitrary 
or $B$ is Morita equivalent to $\SD(3\mathcal{A})_1$ and $V$ is such that its radical quotient 
(resp. socle) is isomorphic to $T_1$. Then Claim 3 follows by using the decomposition matrix of 
$B$ together with \cite[Prop. (23.7)]{CR}. 

Suppose next that $B$ is Morita equivalent to one of the algebras in
$\{\SD(2\mathcal{A})_1(c),\Q(2\mathcal{A})(c)\}$. If
the radical quotient (resp. socle) of $V$ is isomorphic to $T_1$, then 
$\Omega(\overline{U'})\cong Z_{001}$ (resp. $\Omega^{-1}(\overline{U'})\cong Z_{100}$) where
$$Z_{001}=
\begin{array}{c@{}c@{}c@{}c} &&&T_1\\T_1&&T_0\\&T_0\\&T_1
\end{array},\qquad \qquad Z_{100}=
\begin{array}{c@{}c@{}c@{}c}&T_1\\&T_0\\T_1&&T_0\\&&&T_1\end{array}$$
and $Z_{001}$ is a submodule of $P_{T_1}$ and $Z_{100}$ is a quotient module of 
$P_{T_1}$. 
Using \cite[Lemma 2.3.2]{3sim}, it follows that $Z_{001}$ and 
$Z_{100}$ each have a lift over $W$ such that the $F$-character of this lift is
$\chi_3+\chi_4+\chi_6$ (using the notation in Figure \ref{fig:decompSD(2A)}),
which implies Claim 3 in this case.

Finally, suppose that $B$ is Morita equivalent to $\SD(3\mathcal{A})_1$ and $u=2$ and $v=1$,
or $B$ is Morita equivalent to $\Q(3\mathcal{A})_2$ and $\{u,v\}=\{1,2\}$.
If $V$ is such that its radical quotient (resp. socle) is isomorphic to $T_u$ then 
$\Omega(\overline{U'})=Z_{0v0u}$ (resp. $\Omega^{-1}(\overline{U'})=Z_{u0v0})$ where
$$Z_{0v0u}=
\begin{array}{c@{}c@{}c@{}c@{}c} &&&&T_u\\&&&T_0\\T_u&&T_v\\&T_0\\&T_u
\end{array},\qquad \qquad Z_{u0v0}=
\begin{array}{c@{}c@{}c@{}c@{}c}&T_u\\&T_0\\T_u&&T_v\\&&&T_0\\
&&&&T_u\end{array}$$
and $Z_{0v0u}$ is a submodule of $P_{T_u}$ and $Z_{u0v0}$ is a quotient module of 
$P_{T_u}$.
Using \cite[Lemma 2.3.2]{3sim}, it follows that $Z_{0v0u}$ and 
$Z_{u0v0}$ each have a lift over $W$ such that the $F$-character of this lift is
$\chi_2+\chi_3+\chi_6$ if $B$ is Morita equivalent to $\SD(3\mathcal{A})_1$ 
(using the notation in Figure \ref{fig:decompSD(3A)}) and
the $F$-character of this lift is
$\chi_2+\chi_4+\chi_6$ (resp. $\chi_2+\chi_3+\chi_7$) if $B$ is Morita equivalent to 
$\Q(3\mathcal{A})_2$ and $u=1$ (resp. $u=2$) (using the notation in Figure 
\ref{fig:decompQ(3A)}). This completes the proof of Claim 3.

\medskip

\textit{Claim $4$.} The universal deformation ring $R(G,V)$ is as stated in Theorem \ref{thm:main}. 

\medskip

\textit{Proof of Claim $4$.}
In all cases, it follows by Lemma \ref{lem:quaternionargument}
that $U'$ from Claim 3 is an $R'G$-module. More precisely, there exists a $WG$-module 
endomorphism $\alpha$ of $U'$ such that $W[\alpha]\cong R'$.
By Claim 2, we have $\mathrm{End}_{kG}(\overline{U'})\cong k[t]/(t^{2^{n-2}-1})\cong R'/2R'$.
Moreover, since $\overline{U'}$ is a lift of $V$ over $k[t]/(t^{2^{n-2}-1})$, $\overline{U'}$
is free as a module for $\mathrm{End}_{kG}(\overline{U'})$ of rank 
$\mathrm{dim}_k\,V=\mathrm{deg}(\chi_{5,1})$.
Hence it follows by Lemma \ref{lem:quaternionargument} that 
$\mathrm{End}_{WG}(U')=W[\alpha]\cong R'$ and 
$U'$ is free as a module for $\mathrm{End}_{WG}(U')$.

In other words,  $U'$ defines a lift of $V$ over $R'$.
Let $\tau:R(G,V)\to R'$ be the unique continuous $W$-algebra homomorphism
relative to the lift defined by $U'$. Since $R'/(\mathfrak{m}_{R'}^2+2R')\cong k[t]/(t^2)$,
$\tau$ is surjective if and only if  $R'/(\mathfrak{m}_{R'}^2+2R')\otimes_{R'}U'$ does not define the
trivial lift of $V$ over $k[t]/(t^2)$. However,
$$R'/(\mathfrak{m}_{R'}^2+2R')\otimes_{R'}U'\cong U'/(\alpha^2(U')+2U')\cong 
\overline{U'}/\overline{\alpha}^2(\overline{U'})\cong \overline{U'}/t^2\,\overline{U'}$$
is an indecomposable $kG$-module, which implies that this does not define the
trivial lift of $V$ over $k[t]/(t^2)$. Hence $\tau$ is surjective and induces a surjective 
$k$-algebra homomorphism $$\overline{\tau}:R(G,V)/2R(G,V)\to R'/2R'.$$

Suppose first that $V$ does not correspond to a 3-tube. 
Then $R(G,V)/2R(G,V)$ and $R'/2R'$ are isomorphic
and finite dimensional over $k$, which implies that $\overline{\tau}$ is an isomorphism. 
Because $R'$ is a free $W$-module of finite rank, 
it follows that $\tau$ is an isomorphism.
By Lemma \ref{lem:semidihedraltrick}, $R'$ is isomorphic to a subquotient ring of $WD$.
This proves Claim 4 and
completes the proof of Theorem \ref{thm:main} if $V$ does not correspond to a 3-tube.

Suppose next that $V$ corresponds to a 3-tube. Then the universal mod 2 deformation ring
$R(G,V)/2R(G,V)$ is isomorphic to $k[t]/(2^{n-2})$ and the universal mod 2 deformation is
isomorphic to $\overline{U}$. 
By \cite[Lemma 2.3.3]{3sim}, it follows that $R(G,V)\cong W [[t]]/(q_n(t)(t - 2\mu), a 2^mq_n(t))$ 
for certain $\mu\in W$, $a\in \{0, 1\}$ and $0 < m \in \mathbb{Z}$. If $a = 0$ then 
$R(G,V)\cong W [[t]]/(q_n(t)(t -2\mu))$ is free over $W$. If $a = 1$ then 
$R(G,V)/2^m\,R(G,V) \cong(W/2^mW )[[t]]/(q_n(t)(t - 2\mu))$ is free over 
$W/2^mW$. Therefore it follows that if $a = 0$ (resp. $a = 1$), then there is a lift of 
$\overline{U}$, when regarded as a $kG$-module, over $W$ (resp. over 
$W/2^mW$). But by Proposition \ref{prop:3tubes} we have
$R(G, \overline{U}) \cong k$,
which means we must have $a = 1$ and $m = 1$. Since $V$ corresponds to a 3-tube, 
$D$ is dihedral or semidihedral. Hence by Lemma \ref{lem:semidihedraltrick}, 
the ring $W [[t]]/(t\, q_n(t), 2\, q_n(t))$ is isomorphic to a subquotient ring of $WD$.
This proves Claim 4 and
completes the proof of Theorem \ref{thm:main} if $V$ corresponds to a 3-tube. 
\end{proof}


\vspace{.5cm}\noindent
Frauke M. Bleher\\
Department of Mathematics\\
University of Iowa\\
Iowa City, IA 52242-1419\\
U.S.A.\\[.1cm]
{frauke-bleher@uiowa.edu}

\newpage
\section*{Appendix: Decomposition matrices for the algebras in Section \ref{s:tame}}

\begin{figure}[ht] \caption{\label{fig:decompD(2A)} The decomposition matrix for blocks 
of type  $\D(2\mathcal{A})$ or $\SD(2\mathcal{A})_2(c)$.}
$$\begin{array}{ccc}
&\begin{array}{c@{}c}\varphi_0\,&\,\varphi_1\end{array}\\[1ex]
\begin{array}{c}\chi_1\\ \chi_2\\ \chi_3 \\ \chi_4\\ \chi_{5,i}\end{array} &
\left[\begin{array}{cc}1&0\\1&0\\1&1\\1&1\\2&1\end{array}\right]
&\begin{array}{c}\\ \\ \\ \\1\le i\le 2^{n-2}-1\end{array}
\end{array}$$
\end{figure}

\begin{figure}[ht] \caption{\label{fig:decompSD(2A)} The decomposition matrix for blocks 
of type $\SD(2\mathcal{A})_1(c)$ or $\Q(2\mathcal{A})(c)$.}
$$\begin{array}{ccc}
&\begin{array}{c@{}c}\varphi_0\,&\,\varphi_1\end{array}\\[1ex]
\begin{array}{c}\chi_1\\ \chi_2\\ \chi_3 \\ \chi_4\\ \chi_{5,i}\\ \chi_6\end{array} &
\left[\begin{array}{cc}1&0\\1&0\\1&1\\1&1\\2&1\\0&1\end{array}\right]
&\begin{array}{c}\\ \\ \\ 1\le i\le 2^{n-2}-1\\ \end{array}
\end{array}$$
\end{figure}

\begin{figure}[ht] \caption{\label{fig:decompD(2B)} The decomposition matrix for blocks 
of type $\D(2\mathcal{B})$ or $\SD(2\mathcal{B})_1(c)$.}
$$\begin{array}{ccc}
&\begin{array}{c@{}c}\varphi_0\,&\,\varphi_1\end{array}\\[1ex]
\begin{array}{c}\chi_1\\ \chi_2\\ \chi_3 \\ \chi_4\\ \chi_{5,i}\end{array} &
\left[\begin{array}{cc}1&0\\1&0\\1&1\\1&1\\0&1\end{array}\right]
&\begin{array}{c}\\ \\ \\ \\1\le i\le 2^{n-2}-1\end{array}
\end{array}$$
\end{figure}

\begin{figure}[ht] \caption{\label{fig:decompSD(2B)} The decomposition matrix for blocks 
of type $\SD(2\mathcal{B})_2(c)$ or $\Q(2\mathcal{B})_1(c)$.}
$$\begin{array}{ccc}
&\begin{array}{c@{}c}\varphi_0\,&\,\varphi_1\end{array}\\[1ex]
\begin{array}{c}\chi_1\\ \chi_2\\ \chi_3 \\ \chi_4\\ \chi_{5,i}\\ \chi_6\end{array} &
\left[\begin{array}{cc}1&0\\1&0\\1&1\\1&1\\0&1\\2&1\end{array}\right]
&\begin{array}{c}\\ \\ \\ 1\le i\le 2^{n-2}-1\\\end{array}
\end{array}$$
\end{figure}

\begin{figure}[ht] \caption{\label{fig:decompSD(2B)weird} The decomposition matrix for blocks 
of type $\SD(2\mathcal{B})_4(c)$ or $\Q(2\mathcal{B})_2(p,a,c)$.}
$$\begin{array}{ccc}
&\begin{array}{c@{}c}\varphi_0\,&\,\varphi_1\end{array}\\[1ex]
\begin{array}{c}\chi_1\\ \chi_2\\ \chi_3 \\ \chi_4\\ \chi_{5,i}\end{array} &
\left[\begin{array}{cc}1&0\\1&0\\0&1\\0&1\\1&1\end{array}\right]
&\begin{array}{c}\\ \\ \\ \\1\le i\le 2^{n-2}-1\end{array}
\end{array}$$
\end{figure}

\begin{figure}[ht] \caption{\label{fig:decompD(3A)} The decomposition matrix for blocks of type 
$\D(3\mathcal{A})_1$.}
$$\begin{array}{ccc}
&\begin{array}{c@{}c@{}c}\varphi_0\,&\,\varphi_1\,&\,\varphi_2\end{array}\\[1ex]
\begin{array}{c}\chi_1\\ \chi_2\\ \chi_3 \\ \chi_4\\ \chi_{5,i}\end{array} &
\left[\begin{array}{ccc}1&0&0\\1&1&1\\1&0&1\\1&1&0\\2&1&1\end{array}\right]
&\begin{array}{c}\\ \\ \\ \\1\le i\le 2^{n-2}-1\end{array}
\end{array}$$
\end{figure}

\begin{figure}[ht] \caption{\label{fig:decompSD(3A)} The decomposition matrix for blocks
of type $\SD(3\mathcal{A})_1$.}
$$\begin{array}{ccc}
&\begin{array}{c@{}c@{}c}\varphi_0\,&\,\varphi_1\,&\,\varphi_2\end{array}\\[1ex]
\begin{array}{c}\chi_1\\ \chi_2\\ \chi_3 \\ \chi_4\\ \chi_{5,i}\\ \chi_6\end{array} &
\left[\begin{array}{ccc}1&0&0\\1&1&1\\1&0&1\\1&1&0\\2&1&1\\0&0&1\end{array}\right]
&\begin{array}{c}\\ \\ \\ 1\le i\le 2^{n-2}-1\\\end{array}
\end{array}$$
\end{figure}

\begin{figure}[ht] \caption{\label{fig:decompQ(3A)} The decomposition matrix for blocks
of type  $\Q(3\mathcal{A})_2$.}
$$\begin{array}{ccc}
&\begin{array}{c@{}c@{}c}\varphi_0\,&\,\varphi_1\,&\,\varphi_2\end{array}\\[1ex]
\begin{array}{c}\chi_1\\ \chi_2\\ \chi_3 \\ \chi_4\\ \chi_{5,i}\\ \chi_6\\ \chi_7\end{array} &
\left[\begin{array}{ccc}1&0&0\\1&1&1\\1&0&1\\1&1&0\\2&1&1\\0&1&0\\0&0&1\end{array}\right]
&\begin{array}{c}\\ \\ \\ 1\le i\le 2^{n-2}-1\\ \\ \end{array}
\end{array}$$
\end{figure}

\begin{figure}[ht] \caption{\label{fig:decompD(3B)} The decomposition matrix for blocks 
of type $\D(3\mathcal{B})_1$.}
$$\begin{array}{ccc}
&\begin{array}{c@{}c@{}c}\varphi_0\,&\,\varphi_1\,&\,\varphi_2\end{array}\\[1ex]
\begin{array}{c}\chi_1\\ \chi_2\\ \chi_3 \\ \chi_4\\ \chi_{5,i}\end{array} &
\left[\begin{array}{ccc}1&0&0\\1&1&0\\1&0&1\\1&1&1\\0&1&0\end{array}\right]
&\begin{array}{c}\\ \\ \\ \\1\le i\le 2^{n-2}-1\end{array}
\end{array}$$
\end{figure}

\begin{figure}[ht] \caption{\label{fig:decompSD(3B)1} The decomposition matrix for blocks
of type $\SD(3\mathcal{B})_1$ or $\SD(3\mathcal{D})$.}
$$\begin{array}{ccc}
&\begin{array}{c@{}c@{}c}\varphi_0\,&\,\varphi_1\,&\,\varphi_2\end{array}\\[1ex]
\begin{array}{c}\chi_1\\ \chi_2\\ \chi_3 \\ \chi_4\\ \chi_{5,i}\\ \chi_6\end{array} &
\left[\begin{array}{ccc}1&0&0\\1&1&0\\1&0&1\\1&1&1\\0&1&0\\0&0&1\end{array}\right]
&\begin{array}{c}\\ \\ \\ 1\le i\le 2^{n-2}-1\\\end{array}
\end{array}$$
\end{figure}

\begin{figure}[ht] \caption{\label{fig:decompSD(3B)2} The decomposition matrix for blocks 
of type $\SD(3\mathcal{B})_2$.}
$$\begin{array}{ccc}
&\begin{array}{c@{}c@{}c}\varphi_0\,&\,\varphi_1\,&\,\varphi_2\end{array}\\[1ex]
\begin{array}{c}\chi_1\\ \chi_2\\ \chi_3 \\ \chi_4\\ \chi_{5,i}\\ \chi_6\end{array} &
\left[\begin{array}{ccc}1&0&0\\1&1&0\\1&1&1\\1&0&1\\0&1&0\\2&1&1\end{array}\right]
&\begin{array}{c}\\ \\ \\ 1\le i\le 2^{n-2}-1\\\end{array}
\end{array}$$
\end{figure}

\begin{figure}[ht] \caption{\label{fig:decompQ(3B)} The decomposition matrix for blocks 
of type $\Q(3\mathcal{B})$.}
$$\begin{array}{ccc}
&\begin{array}{c@{}c@{}c}\varphi_0\,&\,\varphi_1\,&\,\varphi_2\end{array}\\[1ex]
\begin{array}{c}\chi_1\\ \chi_2\\ \chi_3 \\ \chi_4\\ \chi_{5,i}\\ \chi_6\\ \chi_7\end{array} &
\left[\begin{array}{ccc}1&0&0\\1&1&0\\1&1&1\\1&0&1\\0&1&0\\0&0&1\\2&1&1\end{array}\right]
&\begin{array}{c}\\ \\ \\ 1\le i\le 2^{n-2}-1\\ \\ \end{array}
\end{array}$$
\end{figure}

\begin{figure}[ht] \caption{\label{fig:decompSD(3C)1} The decomposition matrix for blocks
of type $\SD(3\mathcal{C})_{2,1}$.}
$$\begin{array}{ccc}
&\begin{array}{c@{}c@{}c}\varphi_0\,&\,\varphi_1\,&\,\varphi_2\end{array}\\[1ex]
\begin{array}{c}\chi_1\\ \chi_2\\ \chi_3 \\ \chi_4\\ \chi_{5,i}\\ \chi_6\end{array} &
\left[\begin{array}{ccc}0&1&0\\1&1&0\\1&0&1\\0&0&1\\1&0&0\\1&1&1\end{array}\right]
&\begin{array}{c}\\ \\ \\ 1\le i\le 2^{n-2}-1\\ \end{array}
\end{array}$$
\end{figure}

\begin{figure}[ht] \caption{\label{fig:decompSD(3C)2} The decomposition matrix for blocks
of type $\SD(3\mathcal{C})_{2,2}$.}
$$\begin{array}{ccc}
&\begin{array}{c@{}c@{}c}\varphi_0\,&\,\varphi_1\,&\,\varphi_2\end{array}\\[1ex]
\begin{array}{c}\chi_1\\ \chi_2\\ \chi_3 \\ \chi_4\\ \chi_{5,i}\\ \chi_6\end{array} &
\left[\begin{array}{ccc}0&1&0\\1&0&1\\1&1&0\\0&0&1\\1&1&1\\1&0&0\end{array}\right]
&\begin{array}{c}\\ \\ \\ 1\le i\le 2^{n-2}-1\\ \end{array}
\end{array}$$
\end{figure}

\begin{figure}[ht] \caption{\label{fig:decompSD(3H)1} The decomposition matrix for blocks
of type $\SD(3\mathcal{H})_1$.}
$$\begin{array}{ccc}
&\begin{array}{c@{}c@{}c}\varphi_0\,&\,\varphi_1\,&\,\varphi_2\end{array}\\[1ex]
\begin{array}{c}\chi_1\\ \chi_2\\ \chi_3 \\ \chi_4\\ \chi_{5,i}\\ \chi_6\end{array} &
\left[\begin{array}{ccc}1&0&0\\1&1&1\\0&1&0\\0&0&1\\0&1&1\\1&1&0\end{array}\right]
&\begin{array}{c}\\ \\ \\ 1\le i\le 2^{n-2}-1\\\end{array}
\end{array}$$
\end{figure}

\begin{figure}[ht] \caption{\label{fig:decompSD(3H)2} The decomposition matrix for blocks
of type $\SD(3\mathcal{H})_2$.}
$$\begin{array}{ccc}
&\begin{array}{c@{}c@{}c}\varphi_0\,&\,\varphi_1\,&\,\varphi_2\end{array}\\[1ex]
\begin{array}{c}\chi_1\\ \chi_2\\ \chi_3 \\ \chi_4\\ \chi_{5,i}\\ \chi_6\end{array} &
\left[\begin{array}{ccc}1&0&0\\0&1&0\\1&1&1\\0&0&1\\1&1&0\\0&1&1\end{array}\right]
&\begin{array}{c}\\ \\ \\ 1\le i\le 2^{n-2}-1\\\end{array}
\end{array}$$
\end{figure}

\begin{figure}[ht] \caption{\label{fig:decompD(3K)} The decomposition matrix for blocks
of type $\D(3\mathcal{K})$.}
$$\begin{array}{ccc}
&\begin{array}{c@{}c@{}c}\varphi_0\,&\,\varphi_1\,&\,\varphi_2\end{array}\\[1ex]
\begin{array}{c}\chi_1\\ \chi_2\\ \chi_3 \\ \chi_4\\ \chi_{5,i}\end{array} &
\left[\begin{array}{ccc}1&0&0\\1&1&1\\0&1&0\\0&0&1\\0&1&1\end{array}\right]
&\begin{array}{c}\\ \\ \\ \\1\le i\le 2^{n-2}-1\end{array}
\end{array}$$
\end{figure}

\begin{figure}[ht] \caption{\label{fig:decompQ(3K)} The decomposition matrix for blocks 
of type $\Q(3\mathcal{K})$.}
$$\begin{array}{ccc}
&\begin{array}{c@{}c@{}c}\varphi_0\,&\,\varphi_1\,&\,\varphi_2\end{array}\\[1ex]
\begin{array}{c}\chi_1\\ \chi_2\\ \chi_3 \\ \chi_4\\ \chi_{5,i}\\ \chi_6\\ \chi_7\end{array} &
\left[\begin{array}{ccc}1&0&0\\1&1&1\\0&1&0\\0&0&1\\0&1&1\\1&0&1\\1&1&0\end{array}\right]
&\begin{array}{c}\\ \\ \\ 1\le i\le 2^{n-2}-1\\ \\\end{array}
\end{array}$$
\end{figure}


\begin{thebibliography}{88}

\bibitem{alp} J.~L.~Alperin, Local representation theory. Modular representations as an introduction to the local representation theory of finite groups. Cambridge Studies in Advanced Mathematics, vol. 11, Cambridge University Press, Cambridge, 1986.

\bibitem{diloc} F.~M.~Bleher, Universal deformation rings for dihedral $2$-groups. 
J. London Math Soc. 79, (2009), 225--237. 

\bibitem{3sim} F.~M.~Bleher, Universal deformation rings and dihedral defect groups. 
Trans. Amer. Math. Soc. 361 (2009), 3661--3705. 

\bibitem{3quat} F.~M.~Bleher, Universal deformation rings and generalized quaternion defect 
groups. Adv. Math. 225 (2010), 1499--1522.


\bibitem{bc} F.~M.~Bleher and T.~Chinburg, Universal deformation rings and cyclic blocks. Math. Ann. 318 (2000), 805--836.

\bibitem{bc4.9} F.~M.~Bleher and T.~Chinburg, Universal deformation rings need not be
complete intersections. C. R. Math. Acad. Sci. Paris 342 (2006),  229--232.

\bibitem{bc5} F.~M.~Bleher and T.~Chinburg, Universal deformation rings need not be complete intersections. Math. Ann. 337 (2007),  739--767.

\bibitem{bcs} F.~M.~Bleher, T.~Chinburg and B.~de Smit, Inverse Problems for deformation rings.
Accepted by Trans. Amer. Math. Soc., 2012, {\tt arXiv:1012.1290}


\bibitem{2sim} F.~M.~Bleher, G.~Llosent and J.~B.~Schaefer, Universal deformation rings and 
dihedral blocks with two simple modules.  J. Algebra 345 (2011), 49--71. 

\bibitem{bd} V.~M.~Bondarenko and J.~A.~Drozd, The representation type of finite groups. 
Zap. Nauchn. Sem. Leningrad. Otdel. Mat. Inst. Steklov. (LOMI) 71 (1977), 24--41. 
English translation: J. Soviet Math. 20 (1982), 2515--2528.

\bibitem{brauerordinaryandmodular} R.~Brauer, On the Connection Between the Ordinary and 
The Modular Characters of Groups of Finite Order. Ann. Math. (2) (1942), 926--935.

\bibitem{brauerdarst2} R.~Brauer, Zur Darstellungstheorie der Gruppen endlicher Ordnung. II.
Math. Zeitschr. 72 (1959), 25--46.


\bibitem{brIII} R.~Brauer, Some applications of the theory of blocks of characters of finite 
groups. III.  J. Algebra  3  (1966), 225--255.

\bibitem{brblocks} R.~Brauer, On blocks and sections in finite groups. II. Amer. J. Math.
90 (1968), 895--925.

\bibitem{brIV} R.~Brauer, Some applications of the theory of blocks of characters of finite 
groups. IV.  J. Algebra 17 (1971), 489--521.

\bibitem{brauer2} R.~Brauer, On $2$-blocks with dihedral defect groups. Symposia Mathematica, vol. XIII 
(Convegno di Gruppi e loro Rappresentazioni, INDAM, Rome, 1972), pp. 367--393, Academic Press, 
London, 1974. 

\bibitem{br} S.~Brenner, Modular representations of $p$ groups. J. Algebra 15 (1970) 89--102. 

\bibitem{buri} M.~C.~R.~Butler and C.~M.~Ringel, Auslander-Reiten sequences with few middle terms and 
applications to string algebras. Comm. Algebra 15 (1987), 145--179. 


\bibitem{carl2} J.~F.~Carlson and J.~Th\'{e}venaz, The classification of endo-trivial modules.  
Invent. Math.  158  (2004),  389--411.

\bibitem{carl1.5} J.~F.~Carlson and J.~Th\'{e}venaz, The classification of torsion endo-trivial modules.  
Ann. of Math. (2)  162  (2005),  823--883.

\bibitem{flach} T.~Chinburg, Can deformation rings of group representations not be 
	local complete intersections? In: Problems from the Workshop on Automorphisms of Curves. 
	Edited by Gunther Cornelissen and Frans Oort, with contributions by I. Bouw, T. Chinburg, 
	Cornelissen, C. Gasbarri, D. Glass, C. Lehr, M. Matignon, Oort, R. Pries and S. Wewers.
	Rend. Sem. Mat. Univ. Padova 113 (2005), 129--177.


\bibitem{CR} C.~W.~Curtis and I.~Reiner,  Methods of representation theory. Vols. I and II. With applications to finite groups and orders. John Wiley \& Sons, Inc., New York, 1981 and 1987. 

\bibitem{lendesmit} B.~de Smit and H.~W.~Lenstra, Explicit construction of universal deformation rings. In: Modular Forms and Fermat's Last Theorem (Boston,  1995), Springer-Verlag, Berlin-Heidelberg-New York, 1997, pp. 313--326.

\bibitem{eisele} F.~Eisele, $p$-adic lifting problems and derived equivalences.
 J. Algebra 356 (2012), 90--114. 
 
 \bibitem{erdklein} K.~Erdmann, Blocks whose defect groups are Klein four groups: a correction. 
 J. Algebra 76 (1982), 505--518.

\bibitem{erdsemid} K.~Erdmann, Algebras and semidihedral defect groups I. Proc. London
Math. Soc. (3) 57 (1988), 109--150.

\bibitem{erd} K.~Erdmann, Blocks of Tame Representation Type and Related Algebras. Lecture Notes in Mathematics, vol. 1428, Springer-Verlag, Berlin-Heidelberg-New York, 1990.

\bibitem{fong} P.~Fong, A note on splitting fields of representations of finite groups.  
Illinois J. Math.  7  (1963), 515--520.

\bibitem{hi} D.~Higman, Indecomposable representations at characteristic $p$.
Duke Math. J. 21 (1954), 377--381. 

\bibitem{holm} T.~Holm, Derived equivalence classification of algebras of dihedral, semidihedral, 
and quaternion type. J. Algebra 211 (1999), 159--205.

\bibitem{hup} B.~Huppert, Endliche Gruppen. I. Die Grundlehren der Mathematischen Wissenschaften, Band 134, Springer-Verlag, Berlin-New York, 1967.


\bibitem{krause} H.~Krause, Maps between tree and band modules. J. Algebra 137 (1991), 186--194.

\bibitem{linckel} M.~Linckelmann, A derived equivalence for blocks with dihedral defect groups. J. 
Algebra 164 (1994), 244--255.

\bibitem{linckel1} M.~Linckelmann, The source algebras of blocks with a Klein four defect group, J.
Algebra 167 (1994), 821--854.

\bibitem{linckelfuse} M.~Linckelmann, Introduction to fusion systems. In: Group representation 
theory, EPFL Press, Lausanne, 2007, pp. 79--113.

       
\bibitem{maz1} B.~Mazur, Deforming Galois representations. In: Galois groups over $\mathbb{Q}$ (Berkeley,  1987), Springer-Verlag, Berlin-Heidelberg-New York, 1989, pp. 385--437.
        
\bibitem{olsson} J.~B.~Olsson, On $2$-blocks with quaternion and quasidihedral defect groups.
J. Algebra 36 (1975), 212--241.   


\end{thebibliography}
\end{document}